\documentclass[10pts]{amsart}
\usepackage{ amssymb,latexsym, amscd,pb-diagram}
\usepackage{sseq}
\usepackage[all]{xy}
\usepackage[square, numbers]{natbib}
\usepackage{graphicx}
\usepackage{mathrsfs}
\usepackage{color}
\usepackage{amsmath}
\usepackage{enumitem}
\usepackage{tikz}
\usepackage{dsfont}
\usetikzlibrary{matrix}

 \newtheorem{theorem}{Theorem}[section]
 \newtheorem{cor}[theorem]{Corollary}
  
 \newtheorem{lemma}[theorem]{Lemma}
 \newtheorem{proposition}[theorem]{Proposition}
 \theoremstyle{definition}
 \newtheorem{definition}[theorem]{Definition}
 \theoremstyle{definition}
 \newtheorem{example}[theorem]{Example}
 \newtheorem{rem}[theorem]{Remark}
 \numberwithin{equation}{section}

\newcommand{\ben}{\begin{equation}}
\newcommand{\een}{\end{equation}}

\newcommand{\integer}{\ensuremath{{\mathbb Z}}}

\newcommand{\real}{\ensuremath{{\mathbb R}}}
\newcommand{\complex}{\ensuremath{{\mathbb C}}}

\newcommand{\Tt}{\ensuremath{{\mathbb T}}}
\newcommand{\KU}{\ensuremath{{\mathbb{KU}}}}

\newcommand{\KK}{{\mathcal K}}
\newcommand{\id}{{\text{id}}}

\newcommand{\FF}{{\mathcal F}}

\newcommand{\HH}{\mathcal{H}}

\newcommand{\Hom}{\mathrm{Hom}}

\newcommand{\Maps}{\mathrm{Maps}}

\newcommand{\To}{\longrightarrow}

\newcommand{\gr}{\mathfrak}

\newcommand{\Aut}{\ensuremath{{\mathrm{Aut}}}}

\def\va{\varepsilon}

\newcommand{\ot}{{\otimes}}

\newcommand{\C}{{\mathbb C}}

\newcounter{commentcounter}

\begin{document}

\title[Multiplicative Structures on Twisted Equivariant K-theory]{Multiplicative Structures on the Twisted Equivariant K-theory of Finite Groups}

\thanks{C.G would like to thank the hospitality of the Mathematics Department at MIT where part of this work was carried out. C.G. was partially supported by Vicerrector\'ia de Investigaciones de la Universidad de los Andes. I.G acknowledges and thanks the financial support of the Deutscher Akademischer Austausch Dienst.  B.U. acknowledges and thanks the financial support of the Alexander Von Humboldt Foundation.}
\author{C\'esar Galindo}
\address{Departamento de Matem\'{a}ticas, Universidad de los Andes,
Carrera 1 N. 18A - 10, Bogot\'a, Colombia.  }
\email{
cn.galindo1116@uniandes.edu.co, cesarneyit@gmail.com}
\author{Ismael Guti\'errez}
\address{Departamento de Matem\'{a}ticas, Universidad del Norte, Km.5 V\'ia Puerto Colombia, 
Barranquilla, Colombia.   }
\email{
isgutier@uninorte.edu.co}
\author{Bernardo Uribe}
\address{Departamento de Matem\'{a}ticas, Universidad del Norte, Km.5 V\'ia Puerto Colombia, 
Barranquilla, Colombia.   }
\email{buribe@gmail.com}
\begin{abstract}
Let $K$ be a finite group and let $G$ be a finite group acting on $K$ by automorphisms. In this paper we study two
different but intimately related subjects: on the one side we classify all possible multiplicative and associative structures with which
one can endow the twisted $G$-equivariant K-theory of $K$, and on the other, we classify all possible monoidal structures with which one can
endow the category of twisted and $G$-equivariant bundles over $K$. We achieve this classification by encoding the relevant information
in the cochains of a sub double complex of the  double bar resolution associated to the semi-direct product $K \rtimes G$; we use
known calculations of the cohomology of $K$, $G$ and $K \rtimes G$ to produce concrete examples of our classification.

In the case in which $K=G$ and $G$ acts by conjugation, the multiplication map $G \rtimes G \to G$ is a homomorphism of groups and we define a shuffle homomorphism which realizes this map at the homological level. We show that  the categorical information that defines the Twisted Drinfeld Double can be realized as the dual of the shuffle homomorphism applied to any 3-cocycle of $G$. We use the pullback of the multiplication map in cohomology to classify the possible ring structures that the Grothendieck ring of representations of the Twisted Drinfeld Double may have, and we include concrete examples of this procedure.

\end{abstract}

\maketitle

\section*{Introduction}

The purpose of this work is to investigate a relationship existing among certain tensor categories attached to a semi-direct product of groups (they include the Twisted Drinfeld Double of a discrete group), their fusion algebras and the cohomology groups of the  semi-direct product. These tensor categories, as well as some of our results, are closely related to abelian extensions of Hopf algebras and the cohomological description of ${\rm{Opext}}(\mathbb{C}^G,\mathbb{C}F)$) given by Kac in \cite{Kac-Origonal}, cf. \cite{Masuoka-survey}.

The abelian extension theory of Hopf algebras was generalized to coquasi-Hopf algebras by Masuoka in \cite{Masuoka-Adv}. Some of our results and constructions can be framed in the abelian extension theory of coquasi-Hopf algebras in the particular case where  the matched pair of groups is a semi-direct product. However, our approach to these tensor categories does not follow Masuoka's point of view, instead  we use the concept of pseudomonoids in a suitable 2-monoidal 2-category associated to a group. There are some reasons why we prefer to use the pseudomonoid approach:  First, it exhibits more clearly the relationship between the cohomology of semi-direct products and the multiplicative structures of the $G$-equivariant twisted $K$-theory of a finite group $K$. Second, some terms of a spectral sequence associated to the double complex have a direct interpretation in terms of obstructions and classification of the possible pseudomonoid structures. Third, some of our constructions and results make sense in categories different from the category of sets; in particular, if we change to the Cartesian categories of locally compact topological spaces and change discrete group cohomology to the Borel-Moore cohomology \cite{BM-cohomology}, we have a more general theory where the categories of coquasi-Hopf algebras may not capture all the desired information.

There are two main reasons for our interest in the fusion algebra of the tensor categories defined in this paper. On the one side, semisimple tensor categories can be encoded in a combinatorial structure divided in two parts: the fusion algebra (or the Gorthendick ring of the tensor category) and some non-abelian cohomological information provided by the $F$-matrices of the 6$j$-symbols (see \cite{Turaev-Book}).
And on the other,  these fusion algebras are generalizations of  $^\omega \text{K}([G/G])$, the  $w$-twisted stringy $K$-theory of the groupoid $[G/G]$ (see \cite{Willerton}).

In the case that the semi-direct product is finite, the associated tensor categories are fusion categories, and they belong to a bigger family of fusion categories called group-theoretical fusion categories for which many interesting  results have been established cf. \cite{ENO,Ost}. Since an explicit description of the fusion rules  of the tensor categories studied in this paper, via induction and restriction of projective representations of certain subgroups of $G$ already appear in \cite[Theorem 4.8]{Wither}, our approach focuses in determining the number of fusion category and fusion algebra structures associated to a fixed semi-direct product.
We accomplish this task in several steps. First, we show that the information encoding a pseudomonoid with strict unit in the 2-category of $G$-sets with twists over the group $K$ is equivalent to a 3-cocycle in $Z^3({\rm{Tot}}^*(A^{*,*}(K \rtimes G, \Tt)))$, where the double complex   $A^{*,*}(K \rtimes G, \Tt)$ is the sub double complex without the 0-th row of the double bar resolution $C^{*,*}(K \rtimes G, \Tt)$ and whose 
total cohomology calculates the cohomology of $K \rtimes G$. Second, we show that the information encoded in a
pseudomonoid with strict unit in the 2-category of $G$-sets with twists $\mathcal{K}$, is precisely the information required
to endow the category $Bun_G(\mathcal{K})$ of projective $G$-equivariant complex vector bundles with a monoidal structure; hence
the isomorphism classes of bundles ${\rm{Groth}}(Bun_G(\mathcal{K}))$ becomes a fusion algebra, and this fusion algebra structure
could be alternatively understood as a  twisted $G$-equivariant K-theory ring of  $K$. Third, we construct the twisted $G$-equivariant K-theory of the group $K$ and we show the conditions under which this group could be endowed with a multiplicative structure making it a ring; we define the  group of multiplicative structures by $MS_G(K)$ and we show that this group could be calculated by the use  of a spectral sequence associated to 
the complex ${\rm{Tot}}^*(A^{*,*}(K \rtimes G, \Tt))$. We study the canonical homomorphism
$H^3({\rm{Tot}}^*(A^{*,*}(K \rtimes G, \Tt))) \stackrel{\phi}{\to} MS_G(K)$ and we give an explicit description of its kernel and its cokernel;
a multiplicative structure on $MS_G(K)$ not appearing in the image of $\phi$, endows the twisted $G$-equivariant K-theory of $K$ with 
a ring structure which is not the fusion algebra of any tensor category of the form $Bun_G(\mathcal{K})$, in other words,it is an algebra structure which is not possible to categorify. 

Of particular interest is the case in which $G=K$ and $G$ acts on itself by conjugation. In this case, the cohomological information which was used in \cite{Dijkgraaf} to define the Twisted Drinfeld Double $D^w(G)$ for $w \in Z^3(G,\Tt)$ defines an element in 
$Z^3({\rm{Tot}}^*(A^{*,*}(G \rtimes G, \Tt)))$. The formul\ae  \ defining this 3-cocycle were reminiscent of the formul\ae \ appearing
in \cite[Theorem 5.2]{Eilenberg-MacLane} on the proof of the Eilenberg-Zilber theorem, and we conjectured that there had to exist a way to
define for any cocycle in $Z^n(G,\Tt)$ a $n$-cocycle in $Z^n({\rm{Tot}}^*(A^{*,*}(G \rtimes G, \Tt)))$ having similar properties as the ones defined for $n=3$.  We show in this paper that indeed this is the case and its proof is based on two facts: first that the multiplication map
$\mu : G \rtimes G \to G$, $\mu(k,g)=kg$ is a homomorphism of groups, and second, on a construction of an  explicit {\it{Shuffle homomorphism}} at the chain level, whose dual $\tau^\vee: C^*(G,\Tt) \to {\rm{Tot}}^*(C^{*,*}(G \rtimes G, \Tt))$
applied to $w$ recovers the cocycle defined in \cite{Dijkgraaf}, and moreover that in cohomology equals the pullback of $\mu$, e.i.
$\mu^*=\tau^* : H^*(G,\Tt) \to H^*(G \rtimes G,\Tt).$ Since the group $G \rtimes G$ is isomorphic to the product $G\times G$
we get that the map  $H^3({\rm{Tot}}^*(A^{*,*}(G \rtimes G, \Tt))) \stackrel{\phi}{\to} MS_G(G)$ is surjective,  and since we know
that  the cohomology class of the 3-cocycle that is defined in \cite{Dijkgraaf} could be recovered from the cohomology class of $\mu^* w$,
 we give a simple procedure to determine the fusion algebras of ${\rm{Rep}}(D^w(G))$ which are isomorphic to the $G$-equivariant K-theory ring $\KU_G(G)$; at the end of this work this procedure is exemplified in some  interesting cases.

This paper is organized as follows. In Section \ref{Section_preliminaries} we provide background material on the semi-direct products and the double bar complex associated to its cohomology. In Section \ref{Section: The case of the trivializable semi-direct product.} the Shuffle homomorphism of a trivializable semi-direct product is defined and some of its properties are shown. In Section \ref{section:Definiciones categoricas} the 2-monoidal 2-category of twisted $G$-sets with strict unit and the 2-category  of pseudomonoids in this 2-category are defined and described using the complex $\text{Tot}^*(A^{*,*}(K\rtimes G,\Tt))$. In Section \ref{Section:The monoidal category of equivariant vector bundles on a pseudomonoid} the tensor category of equivariant vector bundles over a group, tensor functors  and monoidal natural isomorphism  associated to the 2-category of pseudomonoids in the 2-category of twisted $G$-sets with strict unit are defined. In Section \ref{Section:twisted $G$-equivariant K-theory ring} the obstruction to the existence of multiplicative structures over the twisted $G$-equivariant K-theory of a group $K$ is described using the spectral sequence associated to the filtration $F^r:=A^{*,*> r}$ of the double complex $A^{*,*}$. In Section \ref{Section: Examples} several concrete examples are completely calculated. We finish with an appendix in Section \ref{Appendix: Relation with (coquasi) bialgebras} in which we give the explicit relation of our tensor categories with  coquasi-bialgebras.

\section{Preliminaries} \label{Section_preliminaries}

\subsection{Semi-direct products} Let $K$ be a discrete group endowed with an action of the discrete
group $G$ defined through a homomorphism $\rho: G \to
{\rm{Aut}}(K)$; for simplicity, for $g \in G$ and $k \in
K$ denote the action by $g (k) : = \rho(g)(k)$.  Denote by $K
\rtimes G$ the group defined by the semi-direct product of $G$ with
$K$; as a set $K
\rtimes G: = K \times G $ and the product structure is defined by $(a,g)(b,h) : =
(a \ g( b), gh)$.

The group $K \rtimes G$ fits in the short exact sequence
\begin{equation}\label{sucesion semidirect}
1 \to K \to K \rtimes G \stackrel{\pi_2}{\to} G \to 1
\end{equation}
and we say that $K \rtimes G$ is isomorphic to another split extension
$$ 1 \to K \to E \stackrel{p}{\to} G \to 1$$
whenever there is an isomorphism $\psi: K \rtimes G \stackrel{\cong}{\to} E$ such that $\pi_2 = p \circ \psi$.

 Recall that ${\rm{Inn}}(K)$ is the group of inner automorphisms of the group $K$, i.e. the automorphisms of $K$ induced by conjugation, and that it fits in the short exact sequences
$$1 \to Z(K) \to K \stackrel{\tau}{\to}  {\rm{Inn}}(K) \to 1 $$
$$1 \to {\rm{Inn}}(K) \to {\rm{Aut}}(K) \to  {\rm{Out}}(K) \to 1 $$
where $Z(K)$ denotes the center of $K$ and ${\rm{Out}}(K)$ denotes the group of outer automorphisms of $K$.

\begin{proposition} \label{pro:trivial_semi-direct_product_KxG}
The exact sequence \eqref{sucesion semidirect} of the semi-direct product $K \rtimes G$ is isomorphic to the trivial exact sequence of the product $K \times G$ if and only if the image of $\rho$  is in ${\rm{Inn}}(K)$, and  there is a homomorphism $\sigma : G\to K$ such that the following diagram commutes
$$\xymatrix{
G \ar[dr]_\rho \ar[rr]^\sigma  && K \ar[dl]^\tau\\
& {\rm{Inn}}(K) &
}
$$	where the action of $G$ on $K$ is given by the inner automorphisms defined by $\rho$.​
\end{proposition}

\begin{example}
Let $K=G$ and consider the conjugation action of $K$ on itself. In this case $\rho=\tau$ and we can take  $\sigma= {\rm{Id}}_K$. Therefore the map
$$\mu: K \rtimes K \to K, \hspace{0.5cm} (a,g) \mapsto ag$$
is a homomorphism of groups and
$$\mu \times \pi_2 : K \rtimes K \to K \times K,  \hspace{0.5cm} (a,g) \mapsto (ag,g)$$
is an isomorphism.
\end{example}

\begin{rem} \label{rem:map_from_KxG_to_KxK}
Note that whenever we have a homomorphism $\sigma : G \to K$ such that $\rho = \tau \circ \sigma$ then the map
$$\mu: K \rtimes G \to K, \hspace{0.5cm}  (k,g) \mapsto k \sigma(g)$$
becomes a group homomorphism. Moreover, the map
$$K \rtimes G \to K \rtimes K, \hspace{0.5cm} (k,g) \mapsto (k,\sigma(g))$$
is a homomorphism of groups.
\end{rem}

\subsection{Bar resolution}
Let us find an explicit model for the homology of the group $K \rtimes G$. For this, let us first setup the notation for the explicit model for the bar resolution that we will use. 

Take $H$ a discrete group and define the complex $C_*(EH,\integer)$
with
$$C_n(EH, \integer) := \integer H^{\otimes n+1}$$
and with differential $\partial_H : C_{n}(EH,\integer) \to C_{n-1}(EH,\integer)$ defined by the equation on generators
$$\partial_H(h_1, h_2, ..., h_{n+1}) =(h_2, h_3, ..., h_{n+1}) + \sum_{i=1}^{n} (-1)^i (h_1, ...,
h_ih_{i+1}, ... , h_{n+1}).$$
The complex $(C_*(EH,\integer),\partial_H)$ becomes a complex in the category of $\integer H$-modules if we endow
each $C_n(EH,\integer)$ with the left $\integer H$-module structure defined by the equation
$$h \cdot (h_1,..., h_{n+1}) := (h_1,..., h_{n+1}h^{-1}).$$

The augmentation map
$$\epsilon: C_0(EH,\integer) \to \integer, \hspace{0.5cm}  \epsilon(h) = 1$$
is a map of $\integer H$-modules and the complex $C_*(EH,\integer)$ becomes a $\integer H$-free resolution of the trivial $\integer H$-module $\integer$,
$$\epsilon:C_*(EH,\integer) \to \integer.$$

The elements
$$(h_1, h_2, ..., h_n,1)$$
generate the $\integer H$-module $C_n(EH,\integer)$ and therefore we could write them using the ``bar notation"
$$[h_1|h_2|...|h_n]:=(h_1, h_2, ..., h_n,1);$$
the differential $\partial_H$ in this base becomes
\begin{align*}
\partial_H[h_1| h_2 ... |h_n] =[h_2| h_3| ...| h_n] + \sum_{i=1}^{n}  (-1)^i & [h_1| ...|
h_ih_{i+1}| ... | h_n]\\
&+(-1)^{n+1} h_n^{-1}  \cdot [h_1| ... | h_{n-1}].
\end{align*}

For a left $\integer H$-module $W$, the homology groups of $H$ with coefficients in $W$ are defined as
$$H_*(H,W):= H_*( C_*(EH,\integer) \otimes_{\integer H} W)$$
and the cohomology groups of $H$ with coefficients in $W$ as
$$H^*(H,W):= H^*( {\rm{Hom}}_{\integer H}( C_*(EH,\integer), W)).$$

Since we have a canonical isomorphism of $\integer$-modules
$${\rm{Hom}}_{\integer H}( C_n(EH,\integer), W) \cong {\rm{Maps}}(   H^n, W),$$
the cohomological differential $\delta_H$ in terms of the bar notation becomes
\begin{align*}
(\delta_H f)[h_1| h_2 ... |h_n] =f[h_2| h_3| ...| h_n] + \sum_{i=1}^{n}  (-1)^i & f[h_1| ...|
h_ih_{i+1}| ... | h_n]\\
&+(-1)^{n+1} h_n^{-1}  \cdot f[h_1| ... | h_{n-1}].
\end{align*}

\subsection{Cohomology of $K \rtimes G$} Consider the double complex
$$C_*(EG, \integer) \otimes_\integer C_*(EK,\integer)$$ with differentials $\partial_G \otimes 1$ and $1 \otimes \partial_K$.
Denote by $(g_1,...,g_{p+1}||k_1,...,k_{q+1})$ a generator in $C_p(EG, \integer) \otimes_\integer C_q(EK,\integer)$
and define the action of $(k,g) \in K \rtimes G$ by the equation
$$ (k,g) \cdot (g_1,...,g_{p+1}||k_1,...,k_{q+1}) := (g_1,...,g_{p+1}g^{-1}||g(k_1),...,g(k_{q+1}) k^{-1}).$$
A straightforward computation shows that indeed it is an action and therefore  $C_p(EG, \integer) \otimes_\integer C_q(EK,\integer)$
becomes a free $\integer( K \rtimes G)$-module. Since the differentials $\partial_G \otimes 1$ and $1 \otimes \partial_K$
are also maps of $\integer( K \rtimes G)$-modules, we could take the total complex
$${\rm{Tot}}_* (C_*(EG, \integer) \otimes_\integer C_*(EK,\integer))$$
whose degree $n$-component is
$${\rm{Tot}}_n (C_*(EG, \integer) \otimes_\integer C_*(EK,\integer)) :=  \bigoplus_{p+q=n} C_p(EG, \integer) \otimes_\integer C_q(EK,\integer)$$
and whose differential is
$$\partial_G \otimes 1 \oplus (-1)^p 1 \otimes \partial_K$$
thus obtaining a free $\integer(K \rtimes G)$ complex.
Since the homology of the total complex of this double complex is just $\integer$ in degree 0,
$$H_*( {\rm{Tot}}_* (C_*(EG, \integer) \otimes_\integer C_*(EK,\integer)), \partial_G \otimes 1 \oplus (-1)^p 1 \otimes \partial_K)= \integer,$$
we have that 
$${\rm{Tot}}_* (C_*(EG, \integer) \otimes_\integer C_*(EK,\integer)) \stackrel{\epsilon \otimes \epsilon}{\To} \integer$$
is a free $\integer(K \rtimes G)$ resolution of $\integer$.

Making use of the bar notation we take the elements
$$[g_1|...|g_p||k_1|...|k_q]:= (g_1,...,g_p,1||k_1,...,k_q,1)$$
as a set of generators of $C_p(EG, \integer) \otimes_\integer C_q(EK,\integer)$ as a $\integer(K \rtimes G)$-module;
in this base we have the equality
$$(g_1,...,g_p,g||k_1,...,k_q,k) = (k^{-1},g^{-1})\cdot [g_1|...|g_p||g(k_1)|...|g(k_q)].$$

This choice of base provides us with an isomorphism of $\integer$-modules
$${\rm{Hom}}_{\integer(K \rtimes G)}(C_p(EG, \integer) \otimes_\integer C_q(EK,\integer), \integer) \cong {\rm{Maps}}(G^p \times K^q, \integer)$$
that allows us to transport the dual of the differentials $\partial_G \otimes 1$ and $1 \otimes \partial_K$ to ${\rm{Maps}}(G^p \times K^q, \integer)$; the induced differentials will be denote by $\delta_G$ and $\delta_K$. Therefore we obtain:

\begin{definition} \label{definition double complex}
Let $\overline{C}^{p,q}(K \rtimes G, \integer)$, $p,q \geq 0$,  be
the double complex
$$\overline{C}^{p,q}(K \rtimes G, \integer):= \Maps(G^p \times K^q, \integer)$$
with differentials $\delta_G : \overline{C}^{p,q} \to
\overline{C}^{p+1,q}$ and $\delta_K : \overline{C}^{p,q} \to
\overline{C}^{p,q+1}$  defined by the equations
\begin{align*}(\delta_GF)[g_1|...| g_{p+1} || k_1|...|k_q]
= F[g_2|&...|g_{p+1}|| k_1|...|k_q]\\ & + \sum_{i=1}^{p}(-1)^i F[g_1|...|g_ig_{i+1}|
..|g_{p+1}|| k_1|...|k_q ]\\ &+ (-1)^{p+1} F[g_1|...|g_p||
g_{p+1}(k_1)|...|g_{p+1}(k_q) ]  \\
(\delta_KF)[g_1|...|g_p|| k_1| ...| k_{q+1} ]  = 
F[g_1|&...|g_{p}|| k_2 |...| k_{q+1}] \\ &+
\sum_{j=1}^{q}(-1)^j F[g_1|...|g_p|| k_1| ...|k_jk_{j+1}|...| k_{q+1} ]\\
&+ (-1)^{q+1} F[g_1|...|g_p|| k_1| ...| k_{q}],
\end{align*}
where by convention
$G^0 \times K^0 $ is the set with one point.
\end{definition}
Hence we have
\begin{lemma}
The cohomology of the total complex of the double complex
$\overline{C}^{*,*}(K \rtimes G, \integer), \delta_G, \delta_K$ is
the cohomology of the group $K \rtimes G$, i.e.
$$H^* \left( {\rm{Tot}}^*(\overline{C}^{*,*}(K \rtimes G, \integer)),
\delta_G \oplus (-1)^p \delta_K \right) \cong H^*(K \rtimes G,
\integer).$$
\end{lemma}

We can take a smaller double complex, more suited for our work,
which is called the normalized double complex. Let us define it
\begin{definition} \label{definition normalized double complex for
GxK} The normalized double complex of $\overline{C}^{*,*}(K
\rtimes G, \integer), \delta_G, \delta_K$ is the double complex
$${C}^{p,q}(K \rtimes G,
\integer) \subset \overline{C}^{p,q}(K \rtimes G, \integer) $$
consisting of maps $F : G^p \times K^q \to \integer$ such that
$F[g_1|...|g_p|| k_1|..|k_q]= 1$  whenever  $g_i=1$ or $k_j=1$. The
differentials on $C^{p,q}$ are also $\delta_G$ and $\delta_K$. We
setup $C^{0,0}=\integer$.
\end{definition}
It is  known in homological algebra that the normalized
complex of the bar resolution is quasi-isomorphic to the bar
resolution (see page 215 in \cite{Hilton}). Therefore we have
\begin{lemma}
The induced map on total complexes
$${\rm{Tot}}^*({C}^{*,*}(K \rtimes G, \integer)) \to {\rm{Tot}}^*(\overline{C}^{*,*}(K \rtimes G,
\integer))$$ is a quasi-isomorphism. Then
$$H^* \left( {\rm{Tot}}^*({C}^{*,*}(K \rtimes G, \integer)),
\delta_G \oplus (-1)^p \delta_K \right) \cong H^*(K \rtimes G,
\integer).$$
\end{lemma}

We can generalize our definition of the double complex to other
coefficients. Denote by $\Tt$ the group $S^1$ and consider the
exact sequence of coefficients
$$ 0 \To \integer \To \real \To \Tt \To 0.$$

We will denote the complex $C^{p,q}(K \rtimes G, \Tt)$ as the
complex generated by maps $F:G^p \times K^q \to \Tt$ on which $F$ is
$1$ if one of the entries is the identity, and the differentials
$\delta_G$ and $\delta_K$ are the same ones as Definition
\ref{definition double complex} but changing the sums by
multiplications. Note also that we set up $C^{0,0}(K \rtimes G,
\Tt):= \Tt$.

Then
$$H^* \left( {\rm{Tot}}^*({C}^{*,*}(K \rtimes G, \Tt)),
\delta_G  \oplus \delta_K^{(-1)^p} \right) \cong H^*(K \rtimes
G, \Tt).$$

\subsubsection{Decomposition of the cohomology of $K \rtimes G$}
Note that the 0-th row of the double complex is isomorphic to the normalized bar cochain complex of $G$
$$ (C^{*,0}(K \rtimes G, \integer), \delta_G) \cong (C^*(G,\integer),\delta_G),$$ 
and the action of the differential $\delta_K$ on this row is trivial. Therefore, if we define
\begin{definition} \label{def:complex_A} Let
$$A^{*,*}(K \rtimes G , \integer) := C^{*,*>0}(K \rtimes G , \integer)$$
be  the sub double complex of $C^{*,*}(K \rtimes G , \integer)$ with trivial  0-th row, and differentials $\delta_G$ and $\delta_K$.
\end{definition}
Then we obtain
\begin{lemma} \label{lem:splitting_of_C**_in_A**_and_G} There is a canonical isomorphism of double complexes
$$(C^{*,*}(K \rtimes G , \integer), \delta_G \oplus (-1)^p\delta_K) \cong (A^{*,*}(K \rtimes G , \integer), \delta_G \oplus (-1)^p\delta_K) \oplus (C^*(G,\integer),\delta_G),$$
which induces a canonical isomorphism in cohomology
\begin{align*} 
H^*(K \rtimes G, \integer) \cong H^*({\rm{Tot}}^*(A^{*,*}(K \rtimes G , \integer))) \oplus H^*(G,\integer),
\end{align*}
and the respective one with coefficients in $\Tt$,
\begin{align*} 
H^*(K \rtimes G, \Tt) \cong H^*({\rm{Tot}}^*(A^{*,*}(K \rtimes G , \Tt))) \oplus H^*(G,\Tt)
\end{align*}
where $A^{*,*}(K \rtimes G , \Tt) = C^{*,*>0}(K \rtimes G , \Tt)$
\end{lemma}

\section{The case of the trivializable semi-direct product $K \rtimes G$}\label{Section: The case of the trivializable semi-direct product.}
 Whenever the semi-direct product $K \rtimes G$ is isomorphic to $K \times G$ we know by Proposition \ref{pro:trivial_semi-direct_product_KxG} and Remark \ref{rem:map_from_KxG_to_KxK} that there exists a homomorphism $\sigma: G \to K$ such that $K \rtimes G \to K \rtimes K, (k,g) \mapsto (k, \sigma(g))$ is a homomorphism of groups. In this case we have the homomorphisms
$$K \rtimes G \to K \rtimes K \stackrel{\mu}{\to} K, \hspace{0.5cm} (k,g) \mapsto (k, \sigma(g))\mapsto k\sigma(g)$$
which induce the following homomorphism at the level of their homologies
$$H_*(K \rtimes G, \integer) \to  H_*(K \rtimes K, \integer) \stackrel{\mu_*}{\to} H_*(K , \integer).$$
Since the interesting information lies on the homomorphism $\mu_*$ we will investigate its properties.

\subsection{The Shuffle homomorphism}
In what follows we will describe how to obtain the map $\mu_*$ at the chain level using the explicit models for $H_*(K \rtimes K, \integer)$ and $H_*(K , \integer)$ defined previously. Its definition needs some preparation.

Take a base element
$[g_1|...|g_p||k_1|...|k_q] $
in $C_p(EK,\integer) \otimes_\integer C_q(EK,\integer)$ and think of it as a way to represent $\Delta^p \times \Delta^q$ as the product of the simplices $[g_1|...|g_p] \times[k_1|...|k_q ]$ in $BK$. For $\lambda \in {\rm{Shuff}}(p,q)$ a $(p,q)$-shuffle, i.e. an element in the symmetric group ${\mathfrak{S}}_{p+q}$ such that $\lambda(i) <\lambda(j)$ whenever $1 \leq i < j \leq p $ or $p+1 \leq i < j \leq p+q$, we can define an element
$$\lambda [g_1|...|g_p|k_1|...|k_q] \in C_{p+q}(EK,\integer)$$
such that $\lambda [g_1|...|g_p|k_1|...|k_q]: =[s_1 | ... |s_{p+q}]$ with
\begin{align*}s_{\lambda(i)}= \left\{ \begin{matrix}
g_i & \text{if} & i \leq p\\
g_{p+\lambda(i)-i+1} \dots g_{p-1}g_pk_{i-p} (g_{p+\lambda(i)-i+1} \dots g_{p-1}g_p)^{-1} & \text{if} & i > p.
\end{matrix} \right.\end{align*}
From the Eilenberg-Zilber theorem it follows that the set $$\{\lambda [g_1|...|g_p|k_1|...|k_q] \colon \lambda \in {\rm{Shuff}}(p,q) \}$$
is a subdivision in simplices of dimension $p+q$ of the product of the simplices $[g_1|...|g_p]\times[k_1|...|k_q]$.

An equivalent way to see the elements $\lambda [g_1|...|g_p|k_1|...|k_q]$ is the following: a $(p,q)$-shuffle can also be understood as a way to parameterize a lattice path of minimum distance from the point $(0,0)$ to the point $(p,q)$; one moves one unit to the right   in the steps $\lambda(1), ..., \lambda(p)$ and one unit up in the steps $\lambda(p+1), ..., \lambda(p+q)$. We label the horizontal and vertical edges in the rectangular lattice defined by the points $(0,0),(p,0),(p,q),(0,q)$ by the following rule:
\begin{itemize}
\item The horizontal path between $(i-1,j)$ and $(i,j)$ is labeled with $g_i$ independent of $j$.
\item The vertical path between $(p,j-1)$ and $(p,j)$ is labeled with $k_j$ and all the other vertical paths are labeled in such a way that the squares become commutative squares (when thinking of the labels as maps). This implies that the vertical edge from $(i,j-1)$ to $(i,j)$ is labeled with $$g_{i+1} \dots g_{p-1}g_pk_j(g_{i+1} \dots g_{p-1}g_p)^{-1}. $$
\end{itemize}
Since $\lambda $ parameterizes a path in the lattice, the element $\lambda [g_1|...|g_p|k_1|...|k_q]$ encodes the labels that the path $\lambda$ follow in order to go from $(0,0)$ to $(p,q)$.


We can now define the map of complexes which in homology realizes $\mu_*$:

\begin{definition}
Let 
$$\tau : C_p(EK,\integer) \otimes_\integer C_q(EK,\integer) \to C_{p+q}(EK,\integer) 
$$
be the graded homomorphism defined on generators by the equation
\begin{align*}
\tau(g_1, ... g_p, g||k_1,...,k_q,k) := \sum_{\lambda \in {\rm{Shuff}}(p,q)} (gk)^{-1}\cdot  (-1)^{|\lambda|}\lambda [g_1| ...| g_p|gk_1g^{-1}|...|gk_qg^{-1}]
\end{align*}
where $|\lambda|$ denotes the sign of the permutation $\lambda$. This is usually called  the {\it{Shuffle homomorphism}}.
\end{definition}
We claim that the Shuffle homomorphism $\tau$ is an {\it{admissible product}} in the sense described in \cite[IV-55, p. 118]{Brown}
\begin{theorem} \label{thm:shuffle_map}
The graded homomorphism 
$$\tau : {\rm{Tot}}_*(C_*(EK,\integer) \otimes_\integer C_*(EK,\integer)) \to C_{*}(EK,\integer) 
$$
is a chain map that satisfies the equations 
\begin{align*}
\tau \left( (a^{-1},h^{-1}) \cdot (g_1, ... g_p, g||k_1,...,k_q,k) \right) &= (ha)^{-1} \cdot \tau (g_1, ... g_p, g||k_1,...,k_q,k)\\
(\epsilon \otimes \epsilon) (k,g) &= \epsilon(kg).
\end{align*}
Hence it induces a graded homomorphism in homology
$$\tau_*: H_*(K \rtimes K, \integer) \to H_*(K, \integer)$$
which is equal to the one induced by the pushforward $\mu_*$.
\end{theorem}
\begin{proof}
A routine calculation shows that the equations above are satisfied.
The proof of the fact that $\tau$ is a chain map, namely that
$$\partial_K \circ \tau = \tau \circ ( \partial_K \otimes 1 + (-1)^p 1 \otimes \partial_K), $$
is essentially included in the proof of \cite[Theorem 5.2]{Eilenberg-MacLane}; the decomposition
in simplices of dimension $p+q$ of 
the product of the simplices $[g_1|...|g_p]\times [k_1|...|k_q]$ in $BK$ is done by choosing appropriately $p+q$ edges with the use of the 
$(p,q)$-shuffles as follows
$$\tau[g_1|...|g_p||k_1|...|k_q] = \sum_{\lambda \in {\rm{Shuff}}(p,q)} (-1)^{\lambda} \lambda [g_1| ...| g_p|k_1|...|k_q].$$

With this decomposition of $[g_1|...|g_p]\times [k_1|...|k_q]$, its boundary can be calculated as 
$ (\partial_K \times 1 + (-1)^p 1 \times \partial_K) [g_1|...|g_p]\times [k_1|...|k_q]$ or alternatively as
$\partial_K( \tau[g_1|...|g_p||k_1|...|k_q])$; then it follows that $\tau$ is a chain map.

Now, since $\tau$ is a chain map which preserves the module structures, then it induces a chain map at the level of the coinvariants
$${\rm{Tot}}_*(C_*(EK,\integer) \otimes_\integer C_*(EK,\integer))\otimes_{\integer (K \rtimes K)} \integer \to C_{*}(EK,\integer) \otimes_{\integer K} \integer$$
which defines a homomorphism
$$\tau_*: H_*(K \rtimes K, \integer) \to H_*(K, \integer).$$
This homomorphism $\tau_*$ must be equal to the pushforward $\mu_*$ since $\tau$ preserves the module structures defined by $\mu$. 
\end{proof}

If we have a homomorphism of groups $\sigma : G \to K$ for which the map  $\psi: K \rtimes G \stackrel{\cong}{\to} K \times G, (k,g) \mapsto (k,\sigma(g))$ induces an isomorphism of semi-direct products, then we have that the  map 
\begin{align*}
 \tau \circ (\sigma \otimes 1):C_p(EG,\integer) \otimes_\integer C_q(EK,\integer) & \to C_{p+q}(EK,\integer) \\
 [g_1|...|g_p||k_1|...|k_q] & \mapsto \tau[\sigma(g_1)|...|\sigma(g_p)||k_1|...|k_q]
\end{align*}
induces a chain map
$$\tau \circ (\sigma \otimes 1) : {\rm{Tot}}_*(C_*(EG,\integer) \otimes_\integer C_*(EK,\integer)) \to C_{*}(EK,\integer) $$
preserving their respective module structures, and hence inducing a homomorphism
$$(\tau \circ (\sigma \otimes 1))_* : H_*(K \rtimes G, \integer) \to H_{*}(K,\integer) $$
which is equal to the pushforward map of the composition $ \mu \circ \psi$; i.e.
$$(\mu \circ \psi)_*= (\tau \circ (\sigma \otimes 1))_* : H_*(K \rtimes G, \integer) \to H_{*}(K,\integer).$$

\subsection{The dual of the Shuffle homomorphism}

Dualizing the map $\tau $, we obtain a homomorphism
\begin{align*}
\tau^\vee :\overline{C}^{n}(K, \integer))& \to \bigoplus_{p+q=n} \overline{C}^{p,q}(K \rtimes K, \integer)\\
(\tau^\vee  F)[s_1|...|s_n] &\mapsto \bigoplus_{p+q=n}F(\tau[s_1|...|s_p||s_{p+1}|...|s_{p+q}])
\end{align*}
which induces a cochain map
$$\tau^\vee : (\overline{C}^{*}(K, \integer),\delta_K) \to ( {\rm{Tot}}^*(\overline{C}^{*,*}(K \rtimes K, \integer)), \delta_G \oplus (-1)^{p} \delta_K);$$
here we have kept the notation of the differentials as $\delta_G$ and $\delta_K$ in order to avoid confusion.
A straightforward calculation shows that the cochain map $\tau^\vee$ preserves normalized cochains
$$\tau^\vee : ({C}^{*}(K, \integer),\delta_K) \to ( {\rm{Tot}}^*({C}^{*,*}(K \rtimes K, \integer)), \delta_G  \oplus (-1)^{p}  \delta_K)$$
 and therefore it induces a homomorphism at the level of cohomologies
$$\tau^*: H^*(K ; \integer) \to H^*(K\rtimes K, \integer)$$
which by Theorem \ref{thm:shuffle_map} is equal to the pullback of the homomorphism of groups $\mu$. If we bundle up our previous discussion we have

\begin{theorem}\label{Theorem: mu bullback}
The homomorphism in cohomology
$$\tau^* : H^*(K  ,\integer) \to H^*(K \rtimes K,\integer)$$
 defined by the cochain map 
\begin{align*}
\tau^\vee :{C}^{n}(K, \integer))& \to \bigoplus_{p+q=n}{C}^{p,q}(K \rtimes K, \integer)\\
(\tau^\vee  F)[s_1|...|s_n] &\mapsto \bigoplus_{p+q=n}F(\tau[s_1|...|s_p||s_{p+1}|...|s_{p+q}])
\end{align*}
is equal to 
$$ \mu^* : H^*(K  ,\integer) \to H^*(K \rtimes K,\integer)$$
which is the pullback of the group homomorphism $\mu: K \rtimes K \to K, (k,g)\mapsto kg$.
\end{theorem}

\subsection{Further properties of the Shuffle homomorphism}
Consider the homomorphisms
  \begin{align*}
    \iota_K : K & \to K \rtimes K,  &  \iota_G : K & \to K \rtimes
    K \\
    x & \mapsto (x,1_K)  & g & \mapsto (1_K,g).
  \end{align*}
  and note that they fit into the commutative diagram
  $$\xymatrix{ 
K \ar[r]^{\iota_K}  \ar[rd]_\cong& K \rtimes K \ar[d]^\mu & \ar[l]_{\iota_G} \ar[ld]^\cong K\\
& K & 
  }$$
  which induces the commutative diagram
 \begin{eqnarray}
\label{diagram cohomology of K x G on the adjoint
action}\xymatrix{
 & H^*(K, \integer) \ar[ld]_\cong \ar[d]^{\mu^*} \ar[rd]^\cong & \\
 H^*(K, \integer) & H^*(K \rtimes K, \integer) \ar[l]^{\iota_K^*}
 \ar[r]_{\iota_G^*} & H^*(K,\integer).
}\end{eqnarray}

From Lemma \ref{lem:splitting_of_C**_in_A**_and_G} we know that there is a canonical isomorphism $$H^p(K
\rtimes K, \integer) \cong H^p({\rm{Tot}}^*(A^{*,*}(K
\rtimes K, \integer))) \oplus H^p(K,
\integer)$$ and the homomorphism $\iota_G^*$ is precisely
the projection on the second component of the direct
sum $H^p({\rm{Tot}}^*(A^{*,*}(K
\rtimes K, \integer))) \oplus H^p(K,
\integer)$. 

Moreover, by defining the chain map
\begin{align} \label{shuffle_hom_for_A}
\tau_1^\vee :{C}^{n}(K, \integer))& \to \bigoplus_{p+q=n, \  q>0} {A}^{p,q}(K \rtimes K, \integer)\\
\nonumber (\tau_1^\vee  F)[s_1|...|s_n] &\mapsto \bigoplus_{p+q=n, \ q>0}F(\tau[s_1|...|s_p||s_{p+1}|...|s_{p+q}])
\end{align}
we have that the chain map 
$$\tau^\vee: {C}^{*}(K, \integer)) \to {\rm{Tot}}^*( {A}^{*,*}(K \rtimes K, \integer)) \oplus {C}^{*,0}(K \rtimes K, \integer)$$
is isomorphic to the chain map
$$\tau_1^\vee \oplus 1: {C}^{*}(K, \integer)) \to {\rm{Tot}}^*( {A}^{*,*}(K \rtimes K, \integer)) \oplus {C}^{*,0}(K \rtimes K, \integer)$$
and therefore we have that at the cohomological level we obtain the commutative diagram
 \begin{eqnarray} \label{diagram H*A -> H*K}
 \xymatrix{
 & H^p(K, \integer) \ar[ld]_\cong \ar[d]^{\tau^*_1}   \\
 H^p(K, \integer) & H^p({\rm{Tot}}^*( {A}^{*,*}(K \rtimes K, \integer)) )
 \ar[l]^(0.65){\iota_K^*|_{A}}
}\end{eqnarray} for all $p>0$, where  at the cochain level $\iota_K^*|_A$ is simply  the projection map on the 0-th column
$$\iota_K^*|_A : {\rm{Tot}}^*( {A}^{*,*}(K \rtimes K, \integer)) \to  {A}^{*,0}(K \rtimes K, \integer) \cong C^{*>0}(K,\integer).$$
Therefore we make the following definition
\begin{definition} Let
$$B^{*,*}(K \rtimes G , \integer) := C^{*>0,*>0}(K \rtimes G , \integer)$$
be the sub double complex of $C^{*,*}(K \rtimes G , \integer)$ with trivial  0-th row and trivial 0-th column, and differentials $\delta_G$ and $\delta_K$.
\end{definition}

\begin{lemma} \label{lem:iso_A_with_B_and_cohomology_of_K}
The homomorphism
\begin{align*}
H^p(K,\integer) \oplus H^p( {\rm{Tot}}^*( {B}^{*,*}(K \rtimes K, \integer)))   & \to H^p( {\rm{Tot}}^*( {A}^{*,*}(K \rtimes K, \integer)))\\
x \oplus y  & \mapsto \tau_1^*x + y
\end{align*}
is an isomorphism for all $p>0$.
\end{lemma}
\begin{proof}
The short exact sequence of complexes
$$0 \to  {\rm{Tot}}^*( {B}^{*,*}(K \rtimes K, \integer)) \to  {\rm{Tot}}^*( {A}^{*,*}(K \rtimes K, \integer)) \to C^{*>0}(K,\integer) \to 0$$
induces a long exact sequence in cohomology groups 
$$\to H^p({\rm{Tot}}^*( {B}^{*,*}(K \rtimes K, \integer))) \to  H^p({\rm{Tot}}^*( {A}^{*,*}(K \rtimes K, \integer))) \stackrel{\iota_K^*|_A}{\to} H^{p}(K,\integer) \to$$
which splits by diagram \eqref{diagram H*A -> H*K}.
\end{proof}
Therefore we can conclude that there is a canonical splitting of the cohomology of $K \rtimes K$ in terms of the cohomology of $K$ and of the cohomology of the double complex ${B}^{*,*}(K \rtimes K, \integer)$.
\begin{proposition} \label{pro:tau_pi2_inKxK} The homomorphism 
\begin{align*}
H^p(K,\integer) \oplus H^p( {\rm{Tot}}^*( {B}^{*,*}(K \rtimes K, \integer))) \oplus H^p(K,\integer)  & \stackrel{\cong}{\to} H^p(K \rtimes K,\integer)\\
x \oplus y \oplus z & \mapsto \tau_1^*x + y + \pi_2^*z
\end{align*}
is an isomorphism for all $p>0$. Here $\pi_2:K \rtimes K \to K, (a,g) \mapsto g$ denotes the homomorphism induced by the projection on the second coordinate satisfying $\pi_2 \circ \iota_G= 1$.
\end{proposition}

\section{Categorical definitions}\label{section:Definiciones categoricas}

\subsection{2-category of discrete
$G$-sets with twist}

We shall fix a discrete group $G$.  We define the 2-category of discrete
$G$-sets with twist as follows:

\begin{enumerate}[leftmargin=*]
  \item Objects will be called discrete $G$-sets with twist, and they are pairs
  $(X,\alpha)$, where $X$ is a discrete left
  $G$-set and $\alpha\in Z^2_G(X,\Tt)$ is a normalized 2-cocycle in the
  $G$-equivariant complex of $X$,
  i.e. a map $$\alpha : G \times G \times X \to \Tt,$$ such that
  $$\alpha[\tau|\rho ||x] \ \alpha[\sigma\tau|\rho|| x]^{-1}\ \alpha[\sigma|\tau\rho||x]\ \alpha[\sigma|\tau || \rho x]^{-1}=
 1$$ for all $\sigma,\tau, \rho \in G, x\in
  X$.\\ Note that the previous equation is equivalent to the
  equation $\delta_G \alpha= 1$, when we see $\alpha$ as element
  in $C^{2,1}(X \rtimes G, \Tt)$.
  \item Let $(X,\alpha_X),$ $(Y,\alpha_Y)$ be discrete $G$-sets with twist. A 1-cell from
$(X,\alpha_X)$ to  $(Y,\alpha_Y)$, also called a $G$-equivariant
map, is a pair $(L,\beta)$,
$$\xymatrix{
(X,\alpha_X) \ar[r]^{(L,\beta)} &(Y,\alpha_Y) }$$ where
\begin{itemize}
  \item  $L: X\to Y$ is a morphism of $G$-sets,
  \item $\beta\in C^1_G(X,\Tt)$ is a normalized cochain such that $\delta_G(\beta)= L^*(\alpha_Y)(\alpha_X)^{-1}$,
  i.e., a map $$\beta:G\times X\to \mathbb \Tt$$
\end{itemize}such that $$\beta[\tau || x]\ \beta[\sigma\tau |x]^{-1}\ \beta[\sigma || \tau x]=\alpha_Y[\sigma|\tau || L(x)] \
\alpha_X[\sigma|\tau || x]^{-1},$$ for all $\sigma, \tau \in G, x\in
X$.
  \item Given two 1-cells $(L,\beta), (L,\beta'): (X,\alpha_X)\to
  (Y,\alpha_Y)$, a 2-cell $\theta:(L,\beta)  \Rightarrow
  (L,\beta')$
$$\xymatrix{ (X,\alpha_X)\ar@/^2pc/[rr]_{\quad}^{(L,\beta)}="1"
\ar@/_2pc/[rr]_{(L,\beta')}="2" && (Y,\alpha_Y)
\ar@{}"1";"2"|(.2){\,}="7" \ar@{}"1";"2"|(.8){\,}="8" \ar@{=>}"7"
;"8"^{{\theta}} }$$
  is 0-cochain $\theta\in C^0_G(X,\Tt)$ such that
  $\delta_G(\theta)= \beta'\beta^{-1}$, i.e., a map $\theta: X\to \Tt$,
  such that $$\theta[x] \theta[\sigma x]^{-1} =\beta'[\sigma||x] \ \beta[\sigma||x]^{-1}$$ for
all $\sigma\in G, x\in X$.
\end{enumerate}

Let us define the composition on 1-cells and 2-cells in the
following way. Let $(F,\beta_F): (X,\alpha_X)\to (Y,\alpha_Y)$ and
$(G,\beta_G):(Y,\alpha_Y)\to (Z,\alpha_Z)$ two 1-cells,  define
their composition as $$(G,\beta_G)\circ (F,\beta_F)=(G\circ F,
F^*(\beta_G)\beta_F):(X,\alpha_X)\to (Z,\alpha_Z),$$
$$\xymatrix{
(X, \alpha_X) \ar[rr]^{(F,\beta_F)} \ar[dr]_{(G\circ F,
F^*(\beta_G)\beta_F)} && (Y,\alpha_Y)
 \ar[dl]^{(G,\beta_G)} \\
& (Z,\alpha_Z) &}$$
 and if
$\theta: (L,\beta) \Rightarrow (L, \beta')$ and  $\theta':
(L,\beta') \Rightarrow (L, \beta'')$ are 2-cells, their
composition is the product of the maps, namely
$$\theta'\circ \theta= : \theta'\theta: (L,\beta)\Rightarrow (L,\beta'')$$
$$\xymatrix{ (X, \alpha_X) \ar[rr]^(0.3){(L,\beta')}
\ar@/_3pc/[rr]_{(L,\beta'')}_{}="0"
   \ar@/^3pc/[rr]^{(L,\beta)}_{}="2"  & \ar@{=>}"2";{} ^{\theta} \ar@{=>};"0" ^{\theta'} &
   (Y, \alpha_Y)
  } \qquad = \qquad
\xymatrix{(X, \alpha_X)  \ar@/^2pc/[rr]_{\quad}^{(L,\beta)}="1"
\ar@/_2pc/[rr]_{(L,\beta'')}="2" && (Y, \alpha_Y)
\ar@{}"1";"2"|(.2){\,}="7" \ar@{}"1";"2"|(.8){\,}="8" \ar@{=>}"7"
;"8"^{\theta'\theta} }$$

A straightforward calculation implies that

\begin{lemma}
  The composition of 1-cells and 2-cells satisfy the axioms of a
  2-category.
\end{lemma}

\subsection{Pseudomonoids}

A strict 2-monoidal 2-category is a triple $(\mathcal{B},\boxtimes
, \mathbf{I})$ where $\mathcal{B}$ is a 2-category, $\boxtimes:
\mathcal{B}\times \mathcal{B} \to \mathcal{B}$ is a 2-functor and
$\mathbf{I}$ is an object in $\mathcal{B}$, such that $\boxtimes
\circ (\boxtimes \times Id_{\mathcal{B}})= \boxtimes \circ
(Id_{\mathcal{B}}\times \boxtimes)$ and $\mathbf{I}\boxtimes X=
X\boxtimes \mathbf{I} = X$ for every object in $\mathcal{B}$ (for
more details see \cite{Handbook2}).

\begin{definition} \label{coherentobject}
Given a strict 2-monoidal 2-category $(\mathcal{B},\boxtimes ,
\mathbf{I})$, a
{\bf pseudomonoid } in $\mathcal{B}$ consists of:
\begin{itemize}
 \item an object $\mathbf{C} \in \mathcal{B}$,
\end{itemize}
together with:
\begin{itemize}
 \item a {\bf multiplication} 1-cell $\mathbf{m} : \mathbf{C} \boxtimes \mathbf{C} \to \mathbf{C} $,
 \item an {\bf identity-assigning}
  1-cell $\mathbb{I} : \mathbf{I} \to \mathbf{C} $,
\end{itemize}
together with the following 2-isomorphisms:
\begin{itemize}
\item the {\bf associator}:
\begin{eqnarray}
\xy
 (0,15)*+{\mathbf{C} \boxtimes \mathbf{C} \boxtimes \mathbf{C}}="T";
 (-15,0)*+{\mathbf{C} \boxtimes \mathbf{C}}="L";
 (15,0)*+{\mathbf{C} \boxtimes \mathbf{C}}="R";
 (0,-15)*+{\mathbf{C}}="B";
 (-7,0)*{}="ML";
 (7,0)*{}="MR";
     {\ar^{\id \boxtimes \mathbf{m}} "T";"R"};
     {\ar_{\mathbf{m} \boxtimes \id} "T";"L"};
     {\ar^{\mathbf{m}} "R";"B"};
     {\ar_{\mathbf{m}} "L";"B"};
         {\ar@{=>}^{a} "ML";"MR"};
\endxy \label{diagram associator}
\end{eqnarray}
\item the {\bf left and right unit laws}:
\begin{eqnarray}
\xy
  (-24,5)*+{\mathbf{I} \boxtimes \mathbf{C}}="L";
  (0,5)*+{\mathbf{C} \boxtimes \mathbf{C}}="M";
  (24,5)*+{\mathbf{C} \boxtimes \mathbf{I}}="R";
  (0,-12)*+{\mathbf{C} }="B";
  (-5,2)*{}="TL";
  (5,2)*{}="TR";
  (10,-4)*{}="BR";
  (-10,-4)*{}="BL";
     {\ar^{\mathbb{I} \boxtimes \id} "L";"M"};
     {\ar_{\id \boxtimes \mathbb{I}} "R";"M"};
     {\ar_>>>>>>>{\mathbf{m}} "M";"B"};
     {\ar_{=} "L";"B"};
     {\ar^{=} "R";"B"};
         {\ar@{=>}_{\ell} "TL";"BL"};
         {\ar@{=>}^{r} "TR";"BR"};
\endxy \label{diagram left and right unit laws}
\end{eqnarray}
\end{itemize}
such that the following diagrams commute:
\begin{itemize}
 \item the {\bf pentagon identity} for the associator:
\begin{equation} \label{diagram pentagon}
\def\objectstyle{\scriptstyle}
\def\labelstyle{\scriptstyle}
\xy
 (-35,0)*+{\mathbf{m} \circ (\mathbf{m} \boxtimes {\rm id} ) \circ (\mathbf{m} \boxtimes {\rm id}  \boxtimes {\rm id}) }="TL";
 (-25,-15)*+{ \mathbf{m} \circ (\mathbf{m} \boxtimes {\rm id} ) \circ ({\rm id}  \boxtimes \mathbf{m}  \boxtimes {\rm id})}="BL";
 (35,0)*+{\mathbf{m} \circ ({\rm id} \boxtimes \mathbf{m} ) \circ ({\rm id}  \boxtimes {\rm id} \boxtimes  \mathbf{m})}="TR";
 (25,-15)*+{  \mathbf{m} \circ ({\rm id} \boxtimes \mathbf{m} ) \circ ({\rm id}  \boxtimes \mathbf{m}  \boxtimes {\rm id})}="BR";
 (0,15)*+{\mathbf{m} \circ (\mathbf{m} \boxtimes \mathbf{m})}="T";
     {\ar@{=>}_{\mathbf{m}  \circ (a \boxtimes {\rm id} ) } "TL";"BL"};
     {\ar@{=>}^{a \circ ( \mathbf{m} \boxtimes {\rm id} \boxtimes {\rm id})} "TL";"T"};
     {\ar@{=>}_{a \circ ({\rm id}  \boxtimes \mathbf{m}  \boxtimes {\rm id})} "BL";"BR"};
     {\ar@{=>}_{\mathbf{m}  \circ ({\rm id} \boxtimes a )} "BR";"TR"};
     {\ar@{=>}^{a \circ ({\rm id} \boxtimes {\rm id} \boxtimes \mathbf{m})} "T";"TR"};
\endxy
\end{equation}
where we use the equalities \begin{eqnarray*}({\rm id} \boxtimes
\mathbf{m} ) \circ (\mathbf{m} \boxtimes {\rm id}  \boxtimes {\rm
id})&=&\mathbf{m} \boxtimes \mathbf{m}\\
(\mathbf{m} \boxtimes {\rm id} ) \circ ( {\rm id} \boxtimes {\rm
id} \boxtimes \mathbf{m})&=&\mathbf{m} \boxtimes \mathbf{m}
\end{eqnarray*} in order to compose the upper 1-cells $a \circ ( \mathbf{m} \boxtimes {\rm id} \boxtimes {\rm
id})$ and $a \circ ({\rm id} \boxtimes {\rm id} \boxtimes
\mathbf{m})$, equalities which follow from the fact that
$\boxtimes$ is a 2-functor.

 \item the {\bf triangle identity} for the left and right
unit laws:
\begin{equation} \label{diagram triangle identity}
\def\objectstyle{\scriptstyle}
\def\labelstyle{\scriptstyle}
\xymatrix{
    \mathbf{m} \circ (\mathbf{m} \boxtimes {\rm id}) \circ ({\rm
  id} \boxtimes \mathbb{I} \boxtimes {\rm id})
      \ar@{=>}[rr]^{a \circ ({\rm
  id} \boxtimes \mathbb{I} \boxtimes {\rm id})}
      \ar@{=>}[dr]_{\mathbf{m} \circ (r \boxtimes {\rm id}) }
  &&  \mathbf{m} \circ ({\rm id} \boxtimes \mathbf{m}) \circ ({\rm
  id} \boxtimes \mathbb{I} \boxtimes {\rm id})
      \ar@{=>}[dl]^{\mathbf{m} \circ ({\rm id} \boxtimes \ell)}     \\
  & \mathbf{m}   }
\end{equation}
where we use the fact that $\mathbf{C} \boxtimes \mathbb{I} =
\mathbf{C} = \mathbb{I} \boxtimes \mathbf{C}$ to make sense of the
diagonal arrows.

\end{itemize}
\end{definition}
\begin{rem}
 The 2-category of $G$-sets with twist
has a 2-monoidal structure, where the product of two objects
$(X,\alpha_X)$, $(Y,\alpha_Y)$ is given by
$$(X,\alpha_X)\boxtimes (Y,\alpha_Y)=(X\times Y, \alpha_X\boxtimes \alpha_Y),$$
where $X\times Y$ is a $G$-set with the diagonal $G$-action and
$\alpha_X\boxtimes \alpha_Y:=\pi_1^*(\alpha_X) \pi_2^*(\alpha_Y)$. In
an analogous way we construct the product $\boxtimes$ for 1-cells
and 2-cells. A unit object is any fixed $G$-set with one element and the constant function 1
for 2-cocycle.
\end{rem}

\subsection{Pseudomonoids in the 2-category of $G$-sets with twist} We are interested in studying pseudomonoids in the
2-category of $G$-sets with twists, but in order to get normalized
cocycles we are forced to consider only pseudomonoids where the
{\it identity-assigning} 1-cell is {\bf strict} in the sense that
the cochain (the second component of the 1-cell) is trivial, and
furthermore, that the unit constraints in diagram \eqref{diagram
left and right unit laws} are identities, namely that the diagram
\eqref{diagram left and right unit laws} commutes. We will  call
these {\bf pseudomonoids with strict unit}.

\begin{proposition} \label{proposition pseudomonoid with strict unit 3-cocycle}
A pseudomonoid with strict unit in the 2-category of $G$-sets with
twists is equivalent to the following data:
\begin{itemize}[leftmargin=*]
  \item A monoid $(K,m,1)$, where $K$ is a $G$-set, $m$ is
  $G$-equivariant and $1\in K$ is a $G$-invariant element.
  \item $\alpha \in C^{2,1}(K \rtimes G,\Tt)$, $\beta\in C^{1,2}(K\rtimes G,\Tt)$,
  $\theta \in C^{0,3}(K\rtimes G, \Tt)$ such that $\alpha
  \oplus \beta \oplus \theta$ is a three cocycle in $\left( {\rm{Tot}}({A}^{*,*}(K \rtimes G, \Tt)),
\delta_G \oplus  \delta_K^{{(-1)}^{p}} \right)$ with $A^{*,*}(K \rtimes G, \Tt)$ the double complex introduced in Definition \ref{def:complex_A}.
\end{itemize}
\end{proposition}
\begin{proof}
 A pseudomonoid in the 2-category of $G$-sets is:
\begin{itemize}[leftmargin=*]
 \item[{\it i})]  An object
$\mathbf{C}=(K, \alpha)$ where $K$ is a $G$-set and
$$\alpha:G \times G \times K \to \Tt$$ such that $\alpha$ is
normalized in the components of the group $G$ and that $\delta_G
\alpha = 1$.
 \item[{\it ii})] A multiplication 1-cell $$\mathbf{m}=(m,\beta): (K \times
 K, \alpha \boxtimes \alpha) \to (K,\alpha)$$
 such that $m : K \times K \to K$ is a $G$-equivariant map, and a
 map
 $$\beta: G \times K \times K \to \Tt$$
 satisfying the equation
 \begin{eqnarray}
\label{equation for beta} \delta_G \beta = m^*\alpha \cdot (\alpha
\boxtimes \alpha)^{-1}.
 \end{eqnarray}
\item[{\it iii})] An identity-assigning 1-cell
$$\mathbb{I}=(\mathbf{1}_K, \gamma): (\{*\},1) \to (K, \alpha)$$
where $\mathbf{1}_K : \{*\} \to K$ is a map choosing a
$G$-invariant element $1_K: = \mathbf{1}_K(*)$ in $K$, and
$\gamma:  G \times \{*\} \to \Tt$ is the constant map 1 because we
are only considering pseudomonoids with strict unit. The cochain
condition on $\gamma$ reads
$$(\delta_G\gamma)[g_1|g_2||*]=\alpha[g_1|g_2||1_K]$$
and since $\gamma$ is the constant function it follows that
$\alpha[g_1|g_2||1_K]=1$, namely that $\alpha$ is normalized in the
$K$ variable.
 \item[{\it iv})] The left hand side of diagram \eqref{diagram
left and right unit laws} translates to the diagram
\begin{eqnarray*}
  \xymatrix{ (\{*\} \times K, 1 \boxtimes \alpha) \ar[rrd]_{(\pi_2,1)} \ar[rr]^{(1_K
  \times {\rm{id}}, \gamma \boxtimes 1)}  &&(K \times K, \alpha \boxtimes
  \alpha) \ar[d]^{(m,\beta)} \\
  && (K,\alpha)
  }
\end{eqnarray*}
where $\pi_2$ is the projection on the second component. The
composition of the 1 at the level of the $G$-sets implies the
equation
$$m(1_K,h) = h$$
for any element $h \in K$. Now, the composition of the 1-cells at
the level of the cochains implies the equation
$$ \gamma[g|| *] \cdot \beta[g||1_K|h]=1$$
with $g \in G$ and $h \in K$.
 Since the cochain $\gamma$ is
equal to the constant function 1, we have that $\beta[g||1_K|h]=1$.
Applying the same arguments as above we conclude that
$\beta[g||h|1_K]=1$ and therefore the left and right unit laws
imply that $\beta$ is normalized on the $K$ components, and
moreover that $1_K$ is a unit for the multiplication map $m$.

\item[{\it v})] Diagram \eqref{diagram associator} translates to
the diagram
\begin{eqnarray*}
  \xymatrix{ & (K \times K \times K, \alpha \boxtimes \alpha \boxtimes
   \alpha)  \ar[dl]_{(m \times {\rm id}, \beta \boxtimes 1)} \ar[dr]^{( {\rm id} \times m, 1 \boxtimes \beta)}   &\\
   (K \times K, \alpha \boxtimes
   \alpha) \ar[rd]_{(m, \beta)} \ar@{=>}[rr]^{\theta} && (K \times K,  \alpha \boxtimes
   \alpha) \ar[ld]^{(m, \beta)} \\
    &(K, \alpha) & }
\end{eqnarray*}
whose commutativity at the level of $G$-sets implies that the
multiplication $m:K \times K \to K$ is associative, and at the
level of cochains the diagram implies  the equation
\begin{align} \label{equation relating theta beta}
(\delta_G\theta)[g||h_1|h_2|h_3] = & \\
 \ \beta[g||h_2|h_3] \ & \beta[g||
h_1h_2|h_3]^{-1}\ \beta[g||h_1|h_2h_3] \ \beta[g||h_1|h_2]^{-1}. \nonumber
\end{align}
\item[{\it vi})] Diagram \eqref{diagram pentagon} translates into
the equation
\begin{align} \label{cocycle eqn theta}
\theta[k_1k_2|k_3|k_4] \ \theta[k_1|k_2|k_3k_4] =
\theta[k_1|k_2|k_3] \ \theta[k_1|k_2k_3|k_4] \ \theta[k_2|k_3|k_4]
\end{align}
for all elements $k_1,k_2,k_3,k_4$ in $K$.

\item[{\it vii})] Diagram \eqref{diagram triangle identity}
translates into the equality
\begin{equation} \label{normalization middle theta}
\theta[k_1| 1_K| k_2] = 1
\end{equation}
for all $k_1,k_2 \in K$, because the unit constraints are trivial.
\end{itemize}

From {\it ii)} we have that the multiplication map $m:K\times K
\to K$ is a $G$-equivariant map and {\it v)} tells us that the
multiplication $m$ is associative. From {\it iii)} we know that
$K$ is provided with a $G$ invariant element $1_K$ and {\it iv)}
tells us that this element is a left and right unit for the
multiplication $m$. We have then that $(K,m,1_K)$ is a monoid
endowed with a $G$-action compatible with $m$ and $1_K$.

Since $K$ is a $G$-equivariant monoid with unit, we can use the
notation of Definition \ref{definition double complex} to see that
equation \eqref{equation for beta} can be written as
\[\delta_G(\beta)  \delta_K(\alpha)=1,\] equation
\eqref{equation relating theta beta} becomes
\[ \delta_G(\theta)  \delta_K(\beta)^{-1}=1 \]
 and equation \eqref{cocycle eqn theta} becomes
\[  \delta_K(\theta)=1. \]
We furthermore have that $\delta_G \alpha =1$ by the definition of
an object in the category of  $G$-sets with twist.

 Equation \eqref{normalization middle theta} implies
that $\theta$ is normalized in the variable of the middle; this
fact, together with the fact that $\delta_2 \theta =1$ implies
that $\theta$ is normalized in all the variables since
\[ 1=\delta_K\theta [1_K|1_K|k_1|k_2] = \theta[1_K|k_1|k_2] \]
\[ 1=\delta_K\theta [k_1|k_2|1_K|1_K] = \theta[k_1|k_2|1_K]. \]
From {\it iii)} we know that $\alpha$ is normalized in the $K$
coordinate, and from {\it iv)} we know that $\beta$ is normalized
in the $K$ coordinates. Therefore the maps $\alpha$, $\beta$ and
$\theta$ are normalized in all variables since their normalization
on coordinates of $G$ follow from the
  definition of the 2-category of  $G$-sets with twists.

Summarizing we have that $\alpha \in C^{2,1}(K \rtimes G,\Tt)$,
$\beta \in C^{1,2}(K\rtimes G,\Tt)$ and $\theta \in
C^{0,3}(K\rtimes G, \Tt)$ such that $\alpha
  \oplus \beta \oplus \theta$ is a three cocycle in $\left( {\rm{Tot}}({A}^{*,*}(K \rtimes G, \Tt)),
\delta_G \oplus  \delta_K^{(-1)^{p}} \right)$ because we have that
\begin{align*}
  (\delta_G \oplus \delta_K^{(-1)^{p}})(\alpha
  \oplus \beta \oplus \theta) = & \delta_G (\alpha) \oplus \delta_K
  (\alpha) \delta_G( \beta) \oplus \delta_K( \beta)^{-1} \delta_G
  (\theta) \oplus \delta_K (\theta) \\
  =& 1 \oplus 1 \oplus 1 \oplus 1,
\end{align*} or in a diagram

\begin{tikzpicture}
\matrix [matrix of math nodes,row sep=5mm]
{
 4 &  [5mm]  |(a)| 1  & [5mm]   & [5mm]  & [5mm] & [5mm] \\
3 & |(b)| \theta & |(c)| 1  &  & & \\
2&   & |(d)| \beta & |(e)| 1&  & \\
1& &  & |(f)| \alpha & |(g)|1 \\
& 0 & 1 & 2 & 3&\\
};
\tikzstyle{every node}=[midway,auto,font=\scriptsize]
\draw[thick] (-1.8,-1.7) -- (-1.8,2.2) ;
\draw[thick] (-1.8,-1.7) -- (2.2,-1.7) ;
\draw[-stealth] (b) -- node {$\delta_K$} (a);
\draw[-stealth] (b) -- node {$\delta_G$} (c);
\draw[-stealth] (d) -- node {$(\delta_K)^{-1}$} (c);
\draw[-stealth] (d) -- node {$\delta_G$} (e);
\draw[-stealth] (f) -- node {$\delta_K$} (e);
\draw[-stealth] (f) -- node {$\delta_G$} (g);
\end{tikzpicture}

From the construction above, it is easy to see that if we are
given the $G$ equivariant monoid $(K,m,1_K)$ with unit plus the
cocycle $\alpha \oplus \beta \oplus \theta$, then we can
construct in a unique way a pseudomonoid with strict unit in the
2-category of $G$-sets with twists. This finishes the
proof.\end{proof}

\begin{rem} For a fixed $G$ equivariant monoid $(K,m,1_K)$, the possible 
pseudomonoid structures with strict unit in the
2-category of $G$-sets associated over $(K,m,1_K)$ are classified by elements in
$$Z^3({\rm{Tot}}({A}^{*,*}(K \rtimes G, \Tt))),$$
namely, 3-cocycles in the total complex of ${A}^{*,*}(K \rtimes G, \Tt)$.
\end{rem}

\begin{definition} \label{definition morphism of pseudomonoids}
Given pseudomonoids $(\mathbf{C},\mathbf{m},\mathbb{I},a)$ and
$(\mathbf{C}',\mathbf{m}',\mathbb{I}',a')$ in a 2-category
$\mathcal B$, a {\bf morphism}
$$\mathcal{F}:(\mathbf{C},\mathbf{m},\mathbb{I},a) \To
(\mathbf{C}',\mathbf{m}',\mathbb{I}',a')$$ consists of:
\begin{itemize}
\item a 1-cell $\mathbb{F} : \mathbf{C} \to\mathbf{ C'}$
\end{itemize}
equipped with:
\begin{itemize}
\item a 2-isomorphism
\[
\xy
 (0,15)*+{\mathbf{C} \boxtimes \mathbf{C}}="T";
 (-15,0)*+{\mathbf{C'} \boxtimes \mathbf{C'}}="L";
 (15,0)*+{\mathbf{C}}="R";
 (0,-15)*+{\mathbf{C'}}="B";
 (-7,0)*{}="ML";
 (7,0)*{}="MR";
     {\ar^{\mathbf{m}} "T";"R"};
     {\ar_{\mathbb{F} \boxtimes \mathbb{F}} "T";"L"};
     {\ar^{\mathbb{F}} "R";"B"};
     {\ar_{\mathbf{m'}} "L";"B"};
         {\ar@{=>}^{F_2} "ML";"MR"};
\endxy
\]
\item a 2-isomorphism
\begin{eqnarray} \label{2-isomorphism condition for morphism in
pseudomoids}
 \xy
 (-12,0)*+{\mathbf{C}}="L";
 (12,0)*+{\mathbf{C'}}="R";
 (0,16)*+{\mathbf{I}}="T";
    {\ar_{\mathbb I} "T";"L"};
    {\ar^{\mathbb I'} "T";"R"};
    {\ar_{\mathbb{F}} "L";"R"};
    {\ar@{=>}^{F_0} (-7,2);(4,7)}
 \endxy
\end{eqnarray}
\end{itemize}
such that diagrams expressing the following laws commute:
\begin{itemize}
\item compatibility of $F_2$ with the associator:
\[ \makebox[0em]{
 \xy
   (-55,10)*+{\mathbf{m'}\circ(\mathbf{m'} \boxtimes \id)\circ(\mathbb{F} \boxtimes \mathbb{F} \boxtimes \mathbb{F})}="tl";
    (0,10)*+{\mathbf{m'}\circ(\mathbb{F} \boxtimes \mathbb{F})\circ(\mathbf{m} \boxtimes \id)}="tm";
    (45,10)*+{\mathbb{F}\circ\mathbf{m}\circ(\mathbf{m} \boxtimes \id)}="tr";
    (-55,-10)*+{\mathbf{m'}\circ(\id \boxtimes \mathbf{m'})\circ(\mathbb{F} \boxtimes \mathbb{F} \boxtimes \mathbb{F})}="bl";
    (0,-10)*+{\mathbf{m'}\circ(\mathbb{F} \boxtimes \mathbb{F})\circ(\id \boxtimes \mathbf{m})}="bm";
    (45,-10)*+{\mathbb{F}\circ\mathbf{m}\circ(\id \boxtimes \mathbf{m})}="br";
        {\ar@{=>}^>>>>>>>>>>>{  \mathbf{m'}\circ (F_2 \boxtimes F)} "tl";"tm"};
        {\ar@{=>}^>>>>>>>>>>>{ F_2\circ (\mathbf{m} \boxtimes \id)} "tm";"tr"};
        {\ar@{=>}^{\mathbb{F} \circ a} "tr";"br"};
        {\ar@{=>}_{ a'\circ (\mathbb{F} \boxtimes \mathbb{F} \boxtimes \mathbb{F})} "tl";"bl"};
        {\ar@{=>}_>>>>>>>>>>>{ \mathbf{m'}\circ (\mathbb{F} \boxtimes F_2)} "bl";"bm"};
        {\ar@{=>}_>>>>>>>>>>>{ F_2\circ (\id \boxtimes \mathbf{m})} "bm";"br"};
 \endxy }
\]

\item compatibility of $F_0$ with the left unit law:
\[
 \xy
 (-25,8)*+{\mathbf{m'}\circ(\mathbb{I}' \boxtimes \mathbb{F})}="tl";
 (20,8)*+{\mathbb{F}}="tr";
 (-25,-8)*+{\mathbf{m'}\circ(\mathbb{F} \boxtimes \mathbb{F})\circ(\mathbb{I} \boxtimes \id)}="bl";
 (20,-8)*+{\mathbb{F} \circ\mathbf{m}\circ (\mathbb{I} \boxtimes \id)}="br";
    {\ar@{=>}^{ \ell'\circ \mathbb{F}} "tl";"tr"};
    {\ar@{=>}_{\mathbb{F}\circ \ell} "br";"tr"};
    {\ar@{=>}_{ \mathbf{m'}\circ (F_0 \boxtimes \mathbb{F})} "bl";"tl"};
    {\ar@{=>}_{  F_2 \circ(\mathbb{I} \boxtimes \id)} "bl";"br"};
 \endxy
\]

\item compatibility of $F_0$ with the right unit law:
\[
 \xy
 (-25,8)*+{\mathbf{m'}\circ(\mathbb{F} \boxtimes \mathbb I')}="tl";
 (20,8)*+{\mathbb{F}}="tr";
 (-25,-8)*+{ \mathbf{m'}\circ(\mathbb{F} \boxtimes \mathbb{F})\circ(\id \boxtimes \mathbb I)}="bl";
 (20,-8)*+{\mathbb{F}\circ\mathbf{m}\circ (\id \boxtimes \mathbb I)}="br";
    {\ar@{=>}^{ \mathbb{F}\circ r'} "tl";"tr"};
    {\ar@{=>}_{  \mathbb{F} \circ r}"br";"tr"};
    {\ar@{=>}_{  \mathbf{m}'\circ (\mathbb{F} \boxtimes F_0)} "bl";"tl"};
    {\ar@{=>}_{ F_2\circ (\id \boxtimes \mathbb I)} "bl";"br"};
 \endxy
\]
\end{itemize}
\end{definition}

\begin{definition}
Given two pseudomonoids with strict unit
$\mathcal{K}=(K,m,1,\alpha,\beta,\theta)$ and
$\mathcal{K}'=(K',m',1',\alpha',\beta',\theta')$ in the 2-category
of $G$-sets with twists, a morphism $\mathcal{F} : \mathcal{K} \to
\mathcal{K}'$ is a morphism of pseudomonoids (as in Definition
\ref{definition morphism of pseudomonoids}) such that the
2-isomorphism $F_0$ of diagram \eqref{2-isomorphism condition for
morphism in pseudomoids} is an identity.
\end{definition}

\begin{proposition} \label{proposition morphism of pseudomonoids with strict unit}
A morphism of pseudomonoids with strict unit in the 2-category of
$G$-sets with twists $\mathcal{F} : \mathcal{K} \to \mathcal{K}'$
 consists of the triple $\mathcal{F}=(F, \chi, \kappa)$ with
 $$F: K \to K'$$
 a $G$-equivariant morphism of monoids, and cochains $\chi \in C^{1,1}(K
 \rtimes G, \Tt)$ and $\kappa \in C^{0,2}(K
 \rtimes G, \Tt)$ such that
\[ (\delta_G \oplus \delta_K^{(-1)^{p}})( \chi \oplus \kappa^{-1}) =
F^*\alpha'/\alpha \oplus F^*\beta'/\beta \oplus F^*\theta'/\theta.
\]
\end{proposition}

\begin{proof}
Following Definition \ref{definition morphism of pseudomonoids} a
morphism $\mathcal{F} : \mathcal{K} \to \mathcal{K}'$ consists of:

\begin{itemize}[leftmargin=*]
\item[{\it i)}] A 1-cell $(F, \chi) : (K, \alpha) \to (K', \alpha')$, i.e. a $G$-equivariant map
$F: K \to K'$ and a normalized cochain $\chi \in C^1_G(K, \Tt)$ such that
\begin{equation} \label{equation delta 1 chi} \delta_G \chi = F^*\alpha'/\alpha.\end{equation}
\item[{\it ii)}] A 2-cell $\kappa \in C^0_G(K \times K, \Tt)$
$$\xymatrix{(K \times K, \alpha  \boxtimes \alpha) \ar[r]^(0.6){(m, \beta)} \ar[d]_{(F \times F, \chi \boxtimes \chi)}
& (K, \alpha) \ar[d]^{(F, \chi)} \\
(K' \times K', \alpha'  \boxtimes \alpha') \ar[r]_(0.6){(m',\beta')} \ar@{=>}[ru]^{\kappa} &(K', \alpha')
}$$
such that $$\delta_G \kappa = (\beta \cdot m^* \chi) / ( \chi \boxtimes \chi \cdot F^*\beta').$$
Note that at the level of the $G$ sets the diagram is commutative, therefore $F : K \to K'$ preserves the multiplication;
and since $$\delta_K \chi = \chi \boxtimes \chi \cdot (m^* \chi)^{-1},$$ we can rewrite the equation above as:
\begin{eqnarray} \label{equation cocycle delta 1 kappa delta 2 chi}
\delta_G \kappa \cdot \delta_K \chi = \beta/F^* \beta'.
\end{eqnarray}

\item[{\it iii)}] The commutativity of the diagram \eqref{2-isomorphism condition for morphism in pseudomoids}
\begin{eqnarray*}
 \xy
 (-12,0)*+{{(K, \alpha)}}="L";
 (12,0)*+{{(K',\alpha')}}="R";
 (0,16)*+{{(\{*\},1)}}="T";
    {\ar_{\mathbb I} "T";"L"};
    {\ar^{\mathbb I'} "T";"R"};
    {\ar_{{(F,\chi)}} "L";"R"}
 \endxy
\end{eqnarray*} implies that the map $F:K \to K'$ preserves the unit and that $\chi(g,1)=1$ for any $g \in G$, namely that
$\chi$ is normalized in the $K$ variable and therefore we could say that $\chi \in C^{1,1}(K \rtimes G, \Tt)$.
\item[{\it iv)}] The commutativity of $\kappa$ with the associator. This implies that $F$ preserves the
associativity of the multiplications $m$ and $m'$, and that the following equation is satisfied
$$
\theta'[F(a_1)| F(a_2)| F(a_3)] \  \kappa[a_2|a_3]\ \kappa[a_1| a_2a_3]= \kappa[a_1|a_2]\ \kappa[a_1a_2|a_3]\
 \theta[a_1|a_2|a_3]
$$
for all $a_1,a_2,a_3 \in K$. Note that this last equation can be written as
\begin{eqnarray} \label{equation cocycle delta 2 kappa}
\delta_K\kappa = \theta/F^* \theta'.
\end{eqnarray}
\item[{\it v)}] Compatibility with the left and right units, but as the 2-cells $F_0, l,r, l'$ and $r'$ are identities,
then this implies that $\kappa[1|a]=1= \kappa[a|1]$ and therefore we have that $ \kappa$ is normalized in the $K$ variables
and we can assume that $\kappa \in C^{0,2}(K \rtimes G, \Tt)$.
\end{itemize}

From the previous arguments it follows that $F : K \to K'$ is a $G$-equivariant morphism of monoids, and moreover,
calculating the differential, we have that
\begin{eqnarray*}
  (\delta_G \oplus \delta_K^{(-1)^{p}})( \chi \oplus \kappa^{-1}) & = & \delta_G( \chi) \oplus
     \delta_K ( \chi)^{-1} \delta_G( \kappa)^{-1} \oplus \delta_K( \kappa)^{-1} \\
     & = & F^* \alpha' / \alpha \oplus F^* \beta' / \beta \oplus F^* \theta' / \theta,
\end{eqnarray*}
where the second equality follows from equations \eqref{equation
delta 1 chi}, \eqref{equation cocycle delta 1 kappa delta 2 chi}
and \eqref{equation cocycle delta 2 kappa}. In a diagram

 \begin{tikzpicture}
\matrix [matrix of math nodes,row sep=5mm]
{3 &  [2mm]   |(a)| F^* \theta' / \theta &  [2mm]  |(c)|    &   [2mm]   \\
2& |(b)| \kappa^{-1}  & |(c)| F^* \beta' / \beta& |(e)|    \\
1& &  |(d)| \chi & |(e)| F^* \alpha' / \alpha  \\
& 0 & 1 & 2 \\
};
\tikzstyle{every node}=[midway,auto,font=\scriptsize]
\draw[thick] (-2.1,-1.3) -- (-2.1,2.0) ;
\draw[thick] (-2.1,-1.3) -- (2.1,-1.3) ;
\draw[-stealth] (b) -- node {$\delta_K$} (a);
\draw[-stealth] (b) -- node {$\delta_G$} (c);
\draw[-stealth] (d) -- node {$(\delta_K)^{-1}$} (c);
\draw[-stealth] (d) -- node {$\delta_G$} (e);
\end{tikzpicture}

This finishes the proof.

\end{proof}

\begin{definition}
Given morphisms $\mathcal{F}= ( \mathbb{F},F_0,F_2)$ and $\mathcal{G}=(\mathbb{G}, G_0,G_2)$ from
 $ (\mathbf{C}, \mathbf{m}, \mathbb{I},a)$  to $ (\mathbf{C}', \mathbf{m}', \mathbb{I}',a')$
pseudomonoids in $\mathcal{B}$, a {\bf
2-morphism} $s:\mathcal{F} \to \mathcal{G}$ is a 2-cell  $s : \mathbb{F} \To \mathbb{G}$ in $\mathcal{B}$ such that the
following diagrams commute:
\begin{itemize}
\item compatibility with $F_2$ and $G_2$:
\begin{eqnarray} \label{diagram compatibility F2 G2}
\xymatrix{ \mathbf{m}' \circ(\mathbb{F} \boxtimes \mathbb{F})
  \ar@{=>}[rr]^{ \mathbf{m}'\circ (s \boxtimes s)}
  \ar@{=>}[d]_{F_2}
&& \mathbf{m}'\circ(\mathbb{G} \boxtimes \mathbb{G})
   \ar@{=>}[d]^{G_2}     \\
  \mathbb{F} \circ\mathbf{m}
   \ar@{=>}[rr]^{ s\circ \mathbf{m}}
&&  \mathbb{G} \circ \mathbf{m} }
\end{eqnarray}

\item compatibility with $F_0$ and $G_0$:
\begin{eqnarray} \label{diagram compatibility F0 G0}
 \xy
 (-12,0)*+{\mathbb{F} \circ\mathbb I}="L";
 (12,0)*+{\mathbb{G} \circ \mathbb I}="R";
 (0,14)*+{\mathbb I'}="T";
    {\ar@{=>}_{F_{0}} "L";"T"};
    {\ar@{=>}^{G_{0}} "R";"T"};
    {\ar@{=>}_{s\circ \mathbb I } "L";"R"};
 \endxy
\end{eqnarray}

\end{itemize}
\end{definition}

\begin{proposition} \label{proposition 2-morphism of pseudomonoids}
  Given morphisms of pseudomonoids with strict unit in the 2-category of
$G$-sets with twists $\mathcal{F} =(F, \chi, \kappa)$ and
$\mathcal{F}' =(F, \chi', \kappa')$, with $\mathcal{F}, \mathcal{F}': \mathcal{K} \to \mathcal{K'}$,
a 2-morphism $\gamma:\mathcal{F} \to \mathcal{F}'$ is a cochain $\gamma: C^{0,1}(K \rtimes G,\Tt)$
such that $$(\delta_G \oplus \delta_K^{(-1)^{p}}) (\gamma) = (\chi'/\chi, \kappa/\kappa').$$

\end{proposition}

\begin{proof}
Let $\gamma: (F,\chi) \to (F, \chi')$ be the 2-cell defined by the 2-morphism, then we have that
$$ \delta_G (\gamma) = \chi'/\chi.$$
Now, by diagram \eqref{diagram compatibility F2 G2} we have that
$$\delta_K(\gamma)= \kappa/\kappa',$$
and by diagram \eqref{diagram compatibility F0 G0} we have that $\gamma$ is a normalized cochain.
\end{proof}

\begin{rem} \label{remark double complex A} Propositions \ref{proposition pseudomonoid with strict unit 3-cocycle},
\ref{proposition morphism of pseudomonoids with strict unit} and \ref{proposition 2-morphism of pseudomonoids}
 imply  that the relevant information encoded in cochains for the 2-category of pseudomonoids with strict unit
 of $G$-sets with twists, is
 given by the cochains of the total complex
$$\left( {\rm{Tot}}({A}^{*,*}(K \rtimes G, \Tt)),
\delta_G \oplus  \delta_K^{{(-1)}^{p}} \right)$$
\end{rem}

For a fixed $G$-equivariant monoid $(K,m,1_K)$, Proposition \ref{proposition
pseudomonoid with strict unit 3-cocycle} tells us that the
3-cocycles   $Z^3({\rm{Tot}}^*(A^{*,*}(K \rtimes G, \Tt)))$ are in one to one
correspondence with the set of possible pseudomonoid structures
with strict unit in the 2-category of $G$-sets with twist over
$K$. If we only consider invertible morphisms
of pseudomonoids as defined in Proposition \ref{proposition morphism of pseudomonoids with strict unit}, we may define a groupoid which encodes the equivalence classes of pseudomonoid structures over $K$. Let us be more explicit.

\subsection{Equivalence classes of pseudomonoid structures over a fixed monoid.}
\begin{definition}
Fix a $G$-equivariant monoid $(K,m,1_K)$. Define the groupoid ${\rm{Psdmn}}^G(K)$ whose set of objects
is $Z^3({\rm{Tot}}^*(A^{*,*}(K \rtimes G, \Tt)))$ and whose morphisms are invertible morphisms
of pseudomonoids as defined in Proposition \ref{proposition morphism of
pseudomonoids with strict unit}. The groupoid ${\rm{Psdmn}}^G(K)$  encodes the information of all pseudomonoid strucures
over $K$ and its coarse moduli space $|{\rm{Psdmn}}^G(K)|$, i.e. the set of equivalence classes defined by the morphisms, 
is the set of equivalence classes of pseudomonoid structures on $K$.
\end{definition}
Note that a morphism in ${\rm{Psdmn}}^G(K)$ consists of a triple $(F, \chi,\kappa)$ where  $F:K \to K$ must be a $G$-equivariant automorphism. If we denote by
\begin{align*}
{\rm{Aut}}_G(K) := \{ f \in {\rm{Aut}}(K) \colon f(gk) = gf(k) \text{ for all } g \in G\}
\end{align*}
then we have that the group ${\rm{Aut}}_G(K)$ is isomorphic to the subgroup of ${\rm{Aut}}(K \rtimes G)$ which leaves the $G$  fixed; 
for a $G$-equivariant automorphism $f \in {\rm{Aut}}_G(K)$ we can associate the automorphism $\overline{f} \in {\rm{Aut}}(K \rtimes G)$ by the
equation
$$\overline{f}(k,g) := (f(k),g).$$
Since the automorphism $\overline{f}$ leaves $G$ fixed, then the groups of automorphism act on the double complex $A^{*,*}(K \rtimes G , \integer)$; we claim
\begin{lemma}
The set  of equivalence classes of pseudomonoid structures on $K$ can be described by the quotient
$$|{\rm{Psdmn}}^G(K)| \cong H^3({\rm{Tot}}^*(A^{*,*}(K \rtimes G, \Tt))) /  {\rm{Aut}}_G(K).$$
\end{lemma}

\begin{proof}
We first perform the quotient with the morphisms $(F, \chi,\kappa)$ where $F$ is the identity on $K$; this quotient is precisely
$H^3({\rm{Tot}}^*(A^{*,*}(K \rtimes G, \Tt)))$. Then we see that elements of $H^3({\rm{Tot}}^*(A^{*,*}(K \rtimes G, \Tt)))$ lying on the same orbit of the action of  ${\rm{Aut}}_G(K)$ define equivalent pseudomonoid structures. 
\end{proof}

In particular we may conclude that if
$H^3({\rm{Tot}}^*(A^{*,*}(K \rtimes G, \Tt)))=0$, then all pseudomonoid structures
with strict unit in the 2-category of $G$-sets with twist over $K$
are isomorphic to the trivial one. Since we are interested in finding pseudomonoid structures
with strict unit in the 2-category of $G$-sets with twist over $K$ non isomorphic to the trivial one,
we will calculate the group $H^3({\rm{Tot}}^*(A^{*,*}(K \rtimes G, \Tt)))$, the group ${\rm{Aut}}_G(K)$  and the set
$|{\rm{Psdmn}}^G(K)|$ for some particular examples.

The main tool we will use in order to calculate the group $H^3({\rm{Tot}}^*(A^{*,*}(K \rtimes G, \Tt)))$ will be
the Lyndon-Hochschild-Serre spectral sequence. This spectral sequence can be obtained if the complex
${\rm{Tot}}^*(A^{*,*}(K \rtimes G, \Tt))$ is filtered by the complexes
$$F^n : ={\rm{Tot}}^*(A^{*\geq n,*}(K \rtimes G, \Tt))$$
thus defining a spectral sequence whose second page becomes 
$$\overline{E}_2^{p,q}= H^p(G,H^q(K,\Tt))$$
where $G$ acts on $H^q(K,\Tt)$ through the induced action of $G$ on $K$; note in particular that
$$\overline{E}_2^{0,q}=H^q(K,\Tt)^G \text{ and } \overline{E}_2^{p,0}= H^p(G,\Tt).$$

On the other hand we have that 
$${\rm{Aut}}_G(K) \cong C_{{\rm{Aut}}(K)}(\rho(G)) $$
where $\rho: G \to {{\rm{Aut}}}(K)$ determines the action of $G$, and 
$$   C_{{\rm{Aut}}(K)}(\rho(G)) := \{ f \in {{\rm{Aut}}(K)} \colon [f,\rho(g)]=1 \text{ for all } g \in G \}$$
is the centralizer of $\rho(G)$ on ${\rm{Aut}}(K)$,
consisting of the automorphisms of $K$ which commute with the $G$-action. 
In particular when $G= {{\rm{Aut}}(K)}$ we have that ${{\rm{Aut}}_G(K)}= Z({{\rm{Aut}}(K)})$. Let us see some examples:

\subsubsection{$K=\integer/p$ for prime $p >2$,  and $G={\rm{Aut}}(\integer/p)=\integer/(p-1)$}
Here we have that ${\rm{Aut}}_G(K)= Z({\rm{Aut}}(K)) \cong \integer/(p-1)$ and that
\begin{align*}
H^0(G, H^3(K,\Tt)) & = H^3(K,\Tt)^G \cong (\integer/p)^{\integer/(p-1)} = 0\\
H^1(G, H^2(K,\Tt)) & ={\rm{Hom}}(\integer/(p-1),0)=0\\
H^2(G, H^1(K,\Tt)) &= H^2(\integer/(p-1), \integer/p)= 0
\end{align*}
where the last equality follows from the fact that $H^n(G,M)$ is annihilated by $|G|$ for all $n>0$ \cite[III.10.2]{Brown}.
In this case $H^3({\rm{Tot}}^*(A^{*,*}(K \rtimes G, \Tt)))=0$ and therefore all pseudomonoid structures on $K=\integer/p$ are equivalent
to the trivial one. This example could be generalized as follows:

\subsubsection{Groups with order relatively prime} 
 Let us recall two facts.  First, if $G$ is a finite group of order $m$, $r$ is a positive integer with $(m,r)=1$ and $A^r=0$, then $H^n(K,A)=0$ for all $n$ and all subgroups $K$ of $G$, see \cite[Proposition 1.3.1]{Karpilovsky}. And second, if $e$ is the exponent of $H^2(G,\Tt)$ then $e^2$ divides the order of $G$,  see \cite[Theorem 2.1.5]{Karpilovsky}.

Then for the case on which   $|G|$ is relatively prime to $|K|$ we obtain that
 $$H^1(G,H^2(K,\Tt))=0 \text{ and }H^2(G,H^1(K,\Tt))=0.$$ Therefore we have that 
$H^3({\rm{Tot}}^*(A^{*,*}(K \rtimes G, \Tt)))\cong H^3(K,\Tt)^G$ and 
$$|{\rm{Psdmn}}^G(K)| \cong H^3(K,\Tt)^G/\Aut_G(K).$$

\subsubsection{The dihedral group $D_n$ as  a semi-direct product.} The dihedral group is isomorphic to $\integer/n \rtimes \integer/2$ when $\integer/2$ acts on $\integer/n$ by multiplication of $-1$. Since the induced action of $\integer/2$ on the cohomology ring $H^*(\integer/n,\integer) \cong \integer[x]/\langle nx \rangle$ maps $x \mapsto -x$, we have that $x^2 \mapsto x^2$ and therefore
$H^4(\integer/n,\integer)^{\integer/2}= \integer/n$ and
\[
H^2(\integer/2, H^2(\integer/n,\integer))=H^2(\integer/n,\integer)^{\integer/2} = \left\{\begin{matrix} 
\integer/2 & \text{if} & n \text{ is even}\\
0 & \text{if} & n \text{ is odd}.
\end{matrix}  \right.
\]
Since 
\[
H^4(D_n,\integer)= \left\{ \begin{matrix} 
\integer/2 \oplus \integer/2 \oplus \integer/n & \text{if} & n \text{ is even}\\
\integer/2 \oplus \integer/n & \text{if} & n \text{ is odd}
\end{matrix}  \right.
\]
we know that 
\[
H^3({\rm{Tot}}^*(A^{*,*}(\integer/n \rtimes \integer/2, \Tt)))= \left\{\begin{matrix} 
 \integer/2 \oplus \integer/n & \text{if} & n \text{ is even}\\
 \integer/n & \text{if} & n \text{ is odd},
\end{matrix}  \right.
\]
and therefore 
\[
|{\rm{Psdmn}}^{\integer/2}(\integer/n)|
= \left\{\begin{matrix} 
 (\integer/2 \oplus \integer/n ) / {\rm{Aut}}(\integer/n) & \text{if} & n \text{ is even}\\
( \integer/n ) / {\rm{Aut}}(\integer/n)& \text{if} & n \text{ is odd}.
\end{matrix}  \right.
\]
because in this case ${\rm{Aut}}_{\integer/2}(\integer/n)={\rm{Aut}}(\integer/n)$.

In particular, when $n=4$ we have that ${\rm{Aut}}(\integer/4)=\integer/2$. and therefore the action of ${\rm{Aut}}_{\integer/2}(\integer/4)$  on 
$H^3({\rm{Tot}}^*(A^{*,*}(\integer/4 \rtimes \integer/2, \Tt)))$ is trivial. Hence 
$$|{\rm{Psdmn}}^{\integer/2}(\integer/4)|
= \integer/2 \oplus \integer/4.$$

\subsection{The case of the group acting on itself by conjugation}  \label{subsection:definition_of_alpha_beta_theta_from_w}

 Perhaps the most known pseudomonoid with strict unit in the 2-category of $G$-sets
with twist was introduced by Dijkgraaf, Pasquier and Roche in
\cite[Section 3.2]{Dijkgraaf} while defining the quasi Hopf algebra $D^w(G)$ with
 $w \in Z^3(G; \Tt)$. In the equations (3.2.5) and (3.2.6) of \cite{Dijkgraaf} they defined a 3-cocycle
$\alpha_w \oplus \beta_w \oplus \theta_w \in
Z^3({\rm{Tot}}(A^{*,*}(G \rtimes G, \Tt)))$ by the equations 
\begin{align*}
  \alpha_w[g|h||x]  & := \frac{w[g|h|x] \  w[ghxh^{-1}g^{-1}|
  g|h]}{w[g|hxh^{-1}|h]}\\
  \beta_w[g||x|y] & := \frac{w[g|x|y] \
  w[gxg^{-1}|gyg^{-1}|g]}{w[gxg^{-1}|g|y]}\\
  \theta_w[x|y|z]& := w[x|y|z],
\end{align*}
where $ \alpha_w$ was used to define the algebra law, $\beta_w$ to define the coalgebra law, and $ \theta_w$ encoded the fact that
the coproduct is quasicoassociative (to be precise, in order to get exactly the same formul{\ae} as in \cite{Dijkgraaf}, it is necessary to change $G$ by $G^{op}$).

This quasi Hopf algebra $D^w(G)$ is known as the {\it{Twisted Drinfeld Double}} of $G$ twisted by $w$ (cf. \cite{Willerton}).
Firstly we claim the following.

\begin{lemma}
The 3-cocycle $\alpha_w \oplus \beta_w \oplus \theta_w$ equals $\tau_1^\vee w$, the image 
of $w$ under the restricted shuffle homomorphis $\tau_1^\vee w$ defined in \eqref{shuffle_hom_for_A}.
\end{lemma}
\begin{proof}
From the following diagrams associated to $(\tau_1^\vee w)[g|h|k]$

\begin{tikzpicture}
\matrix [matrix of math nodes,row sep=5mm]
{  |(a)|&  [5mm]   |(b)|  &  [5mm]|(c)| \bullet   \\
  |(d)| \bullet  & |(e)|  \bullet& |(f)| \bullet   \\
}; 
\tikzstyle{every node}=[midway,auto,font=\scriptsize]
\draw[-stealth] (d) -- node [below] {$g$} (e);
\draw[-stealth] (e) -- node  [below]{$h$} (f);
\draw[-stealth] (f) -- node {$k$} (c);
\end{tikzpicture} \hspace{1cm}
\begin{tikzpicture}
\matrix [matrix of math nodes,row sep=5mm]
{  |(a)|&  [5mm]   |(b)| \bullet    &  [5mm]|(c)| \bullet   \\
  |(d)| \bullet  & |(e)|  \bullet& |(f)|  \\
};
\tikzstyle{every node}=[midway,auto,font=\scriptsize]
\draw[-stealth] (d) -- node  [below] {$g$} (e);
\draw[-stealth] (e) -- node {$hkh^{-1}$} (b);
\draw[-stealth] (b) -- node {$h$} (c);
\end{tikzpicture}
\hspace{0.2cm}
\begin{tikzpicture}
\matrix [matrix of math nodes,row sep=5mm]
{  |(a)| \bullet&  [5mm]   |(b)| \bullet    &  [5mm]|(c)| \bullet   \\
  |(d)| \bullet  & |(e)|  & |(f)|  \\
};
\tikzstyle{every node}=[midway,auto,font=\scriptsize]
\draw[-stealth] (d) -- node [left]  {$ghkh^{-1}g^{-1}$} (a);
\draw[-stealth] (a) -- node {$hkh^{-1}$} (b);
\draw[-stealth] (b) -- node {$h$} (c);
\end{tikzpicture}

\begin{tikzpicture}
\matrix [matrix of math nodes,row sep=5mm]
{  |(a)|&  [5mm]   |(b)| \bullet \\
  |(c)|    &  |(d)| \bullet \\ 
  |(e)|  \bullet  & |(f)| \bullet   \\
}; 
\tikzstyle{every node}=[midway,auto,font=\scriptsize]
\draw[-stealth] (e) -- node [above] {$g$} (f);
\draw[-stealth] (f) -- node  [right] {$h$} (d);
\draw[-stealth] (d) -- node [right] {$k$} (b);
\end{tikzpicture} \hspace{1cm}
\begin{tikzpicture}
\matrix [matrix of math nodes,row sep=5mm]
{  |(a)|&  [5mm]   |(b)| \bullet \\
  |(c)|   \bullet   &  |(d)| \bullet \\ 
  |(e)|  \bullet  & |(f)|  \\
}; 
\tikzstyle{every node}=[midway,auto,font=\scriptsize]
\draw[-stealth] (e) -- node [left] {$ghg^{-1}$} (c);
\draw[-stealth] (c) -- node  {$g$} (d);
\draw[-stealth] (d) -- node [right]{$k$} (b);
\end{tikzpicture} \hspace{1cm}
\begin{tikzpicture}
\matrix [matrix of math nodes,row sep=5mm]
{  |(a)|   \bullet  &  [5mm]   |(b)| \bullet \\
  |(c)|   \bullet   &  |(d)|\\ 
  |(e)|  \bullet  & |(f)|  \\
}; 
\tikzstyle{every node}=[midway,auto,font=\scriptsize]
\draw[-stealth] (e) -- node [left] {$ghg^{-1}$} (c);
\draw[-stealth] (c) -- node [left] {$gkg^{-1}$} (a);
\draw[-stealth] (a) -- node [above]{$g$} (b);
\end{tikzpicture}\\
we obtain that $\alpha_w \oplus \beta_w \oplus \theta_w =\tau_1^\vee w$.
\end{proof}
Therefore by Lemma \ref{lem:splitting_of_C**_in_A**_and_G}, Theorem \ref{thm:shuffle_map} and Lemma \ref{lem:iso_A_with_B_and_cohomology_of_K} we get
\begin{proposition}
The cohomology class 
$$[\alpha_w \oplus \beta_w \oplus \theta_w] \oplus [w]\in H^3({\rm{Tot}}^*(A^{*,*}(G \rtimes G, \Tt))) \oplus H^3(G,\Tt)$$
is equal to $\mu^*[w]$ where $\mu: G \rtimes G \to G, (a,g) \mapsto ag$ is the multiplication map. Moreover
we obtain the isomorphism 
\begin{align*}
H^3(G,\Tt) \oplus H^3({\rm{Tot}}^*(B^{*,*}(G \rtimes G, \Tt)))& \to H^3({\rm{Tot}}^*(A^{*,*}(G \rtimes G, \Tt)))\\
[w] \oplus [x] & \mapsto [\alpha_w \oplus \beta_w \oplus \theta_w] +[x].
\end{align*}
\end{proposition}

Now, since $G \rtimes G \cong G \times G$, we have that the Lyndon-Hochschild-Serre spectral sequence collapses at the second page.
And since the action of $G$ on $G$ is given by conjugation, then the action of $G$ on $H^*(G,\Tt)$ is trivial. Hence we have that
 $H^3({\rm{Tot}}(B^{*,*}(G \rtimes G, \Tt)))$ sits in the middle of the
 short exact sequence
  $$ 0 \to H^2(G,{\rm{\Hom}}(G, \Tt)) \to H^3({\rm{Tot}}(B^{*,*}(G \rtimes G, \Tt)))
  \to  {\rm{Hom}}(G, H^2(G,\Tt)) \to 0.$$
  
  Moreover, in the present situation we have $\rho(G)={\rm{Inn}}(G)$, and therefore
  $${\rm{Aut}}_G(G)= C_{{\rm{Aut}}(G)}({\rm{Inn}}(G)),$$
  namely the group of automorphisms of $G$ which commute with all inner automorphisms.
  
  With the previous calculations at hand we can calculate the group
$H^3({\rm{Tot}}(B^{*,*}(G \rtimes G, \Tt)))$ and ${\rm{Aut}}_G(G)$  in some particular examples:

\subsubsection{$G$ simple and non abelian:} \label{G_simple_non_abelian} When $G$ is simple and
  nonabelian, its abelianization $G/[G,G]$ is trivial. Therefore
  ${\rm{Hom}}(G,\Tt)$ and ${\rm{Hom}}(G, H^2(G,\Tt))$ are trivial
  and  hence $$H^3({\rm{Tot}}^*(B^{*,*}(G \rtimes G; \Tt)))=0.$$
 So we have that
 $$|{\rm{Psdmn}}^G(G)| \cong H^3(G,\Tt)/{\rm{Aut}}_G(G).$$

When $G$ is the alternating group $A_n$ we have that ${\rm{Aut}}_{A_n}(A_n)=C_{{\gr{S}}_n}(A_n)=1$ and therefore
$$|{\rm{Psdmn}}^{A_n}(A_n)| \cong H^3(A_n,\Tt), \text{ for } n\neq	6 \text{ and } n >4.$$
 In particular when $n=5$ we have that
  $H^3(A_5,\Tt)=\integer/120$ and hence 
  $$|{\rm{Psdmn}}^{A_5}(A_5)| \cong \integer/120.$$
 
\subsubsection{Binary icosahedral group} \label{G_binary_icosahedral} The binary icosahedral group $\overline{A}_5$ is a subgroup of $SU(2)$ that can be obtained as the pullback of the diagram
$$\xymatrix{\overline{A}_5 \ar[r] \ar[d] & SU(2) \ar[d] \\
A_5 \ar[r] & SO(3)
}$$ 
where $A_5$ embeds in $SO(3)$ as the group of isometries of an icosahedron. This group satisfies $H_1(\overline{A}_5, \integer)=0$, $H_2(\overline{A}_5,\integer)=0$
 and $H_3(\overline{A}_5,\integer)=\integer/120$ (see \cite[Page 279]{AdemMilgram}), and therefore $H^1(\overline{A}_5,\integer)=H^2(\overline{A}_5,\integer)=H^3(\overline{A}_5,\integer)=0$ and $H^4(\overline{A}_5,\integer)=\integer/120$. Hence $H^1(\overline{A}_5,\Tt)=H^2(\overline{A}_5,\Tt)=0$, $H^3(\overline{A}_5,\Tt)=\integer/120$ and $H^3({\rm{Tot}}(B^{*,*}(\overline{A}_5 \rtimes \overline{A}_5; \Tt)))=0.$
 
 In this case ${\rm{Inn}}(\overline{A}_5)\cong A_5$ and ${\rm{Aut}}(\overline{A}_5)\cong{\gr{S}}_5$, therefore ${\rm{Aut}}_{\overline{A}_5}(G)=C_{{\gr{S}}_5}(A_5)=1$ and
 $$|{\rm{Psdmn}}^{A_5}(A_5)| \cong \integer/120.$$
 
A similar argument applies to any superperfect group since by definition they are the ones such that  $H_1(G, \integer)=0$ and $H_2(G,\integer)=0$.

\subsubsection{The dihedral group $G=D_n$ with $n$ odd.} In this
case $${\rm{Hom}}(D_n,\Tt) = \integer/2, \ \ \ H^2(D_n;
\integer/2)=\integer/2 \ \ {\rm{and}} \ \ H^2(D_n; \Tt)=0,$$ hence
$$H^3({\rm{Tot}}^*(B^{*,*}(D_n \rtimes D_n; \Tt)))= H^2(D_n,{\rm{Hom}}(D_n,\Tt)=\integer/2.$$
Since $H^3(D_n; \Tt)= \integer/2 \oplus \integer/n$ we have that
the isomorphism classes of pseudomonoid structures coming from the
Twisted Drinfeld Double construction are $\integer/2 \oplus
\integer/n$ and that there is an independent pseudomonoid
structure which comes from $H^2(D_n, {\rm{Hom}}(D_n;
\Tt))=\integer/2$.  
In this case we have that ${\rm{Aut}}_{D_n}(D_n)=Z(D_n)=1$ and therefore 
$$|{\rm{Psdmn}}^{D_n}(D_n)| \cong( \integer/n \oplus \integer/2 \oplus \integer/2 ).$$
\subsubsection{The symmetric group
$G=\gr{S}_n$ for $n\geq 4$.} From \cite[VI-5]{AdemMilgram} we know that
$${\rm{Hom}}(\gr{S}_n,\Tt) = \integer/2, \ \ \ H^2(\gr{S}_n;
\integer/2)=\integer/2$$ $$H^2(\gr{S}_n; \Tt)=\integer/2 \ \
{\rm{and}} \ \ H^3(\gr{S}_n; \integer/2)=\integer/2 \oplus
\integer/2$$ hence we get the exact sequence
$$ 0 \to \integer/2 \to H^3({\rm{Tot}}^*(B^{*,*}({\gr{S}}_n \rtimes {\gr{S}}_n ; \Tt)))
  \to \integer/2 \to 0.$$
  In particular we could say that the nontrivial
element in $H^2({\gr{S}}_n,{\rm{Hom}}({\gr{S}}_n, \Tt))=\integer/2$
induces a pseudomonoid structure on ${\gr{S}}_n$ which is not
isomorphic to any structure coming from the construction of the
Twisted Drinfeld Double. This follows from the fact that for $n\neq 2$ and $n\neq 6$,
${\rm{Aut}}({\gr{S}}_n)={\rm{Inn}}({\gr{S}}_n)={\gr{S}}_n$ and therefore
${\rm{Aut}}_{{\gr{S}}_n}({\gr{S}}_n)=Z({\gr{S}}_n)=1$ is the trivial
group. Whenever $n=6$ we know that ${\rm{Out}}({\gr{S}}_6) = \integer/2$; nevertheless
${\rm{Aut}}_{{\gr{S}}_6}({\gr{S}}_6)=1$. Therefore we have that
$$|{\rm{Psdmn}}^{{\gr{S}}_n}({\gr{S}}_n)|= H^3({\gr{S}}_n,\Tt) \oplus H^3({\rm{Tot}}^*(B^{*,*}({\gr{S}}_n \rtimes {\gr{S}}_n ; \Tt)))$$
with
$$ H^3({\gr{S}}_n,\Tt) = \left\{ \begin{matrix}
\integer/12 \oplus \integer/2 & \text{if} & n=4,5 \\
\integer/12 \oplus \integer/2 \oplus \integer/2 & \text{if} & n\geq 6
\end{matrix}\right. $$

\subsubsection{$G$ cyclic group.} When $G=\integer/n$ we get that
$$H^3({\rm{Tot}}^*(B^{*,*}( \integer/n \times  \integer/n ; \Tt)))= H^2(\integer/n, {\rm{Hom}}(\integer/n,\Tt))=
\integer/n$$ and $H^3(\integer/n,\Tt)=\integer/n$. The action of ${\rm{Aut}}_{\integer/n}(\integer/n) = \integer/n^\times$, which is the multiplicative group of units in $\integer/n$, on 
$H^1(\integer/n,\Tt)=\integer/n$ is given by multiplication and while on $H^3(\integer/n,\Tt)=\integer/n$ is given by the square of the multiplication; hence we get that
$$|{\rm{Psdmn}}^{\integer/n}(\integer/n)| \cong (\integer/n \oplus \integer/n)/\integer/n^\times$$
where the action is given by
 $$\integer/n^\times \times (\integer/n \oplus \integer/n) \to (\integer/n \oplus \integer/n), \hspace{0.5cm} (a,(x,y))\mapsto (ax,a^2y).$$ 

For example when $n=4$ we have that
$$|{\rm{Psdmn}}^{\integer/4}(\integer/4)| \cong (\integer/4) / (\integer/4^\times)  \times \integer/4.$$

\section{The monoidal category of equivariant vector bundles on a pesudomonoid}\label{Section:The monoidal category of equivariant vector bundles on a pseudomonoid}

Let $G$ be a   group and
$\mathcal{K}=(K,m,1,\alpha,\beta,\theta)$ a pseudomonoid with
strict unit in the 2-category of $G$-sets with twists. 
We define the category $Bun_G(\mathcal K)$ of $G$-equivariant
finite dimensional bundles over $K$ as follows:

An object is a $K$-graded finite dimensional Hilbert space $\mathcal{H}=
\bigoplus_{k\in K}\mathcal H_k$ and a twisted $G$-action
\[\rhd : G \to U(\mathcal H)\] such that
\begin{itemize}
  \item $\sigma \rhd \mathcal H_k = \mathcal H_{\sigma \cdot k}$
  \item $\sigma \rhd (\tau \rhd h_k) = \alpha[\sigma| \tau|| k](\sigma \tau) \rhd h_k$
  \item $e\rhd h = h$
\end{itemize}
for all $\sigma, \tau\in G, k\in K, h_k\in \mathcal H_k$. Morphisms in the category are linear maps that preserve  the grading and the twisted action, i.e., a linear map  $f:\mathcal H\to \mathcal H'$ is a morphism if
\begin{itemize}
  \item $f(\mathcal H_k)\subset \mathcal H'_k$,
  \item $f(\sigma \rhd h)= \sigma \rhd f(h)$
\end{itemize}for all $\sigma \in G, k \in K$ and $h\in \mathcal H$.

We define a monoidal structure on $Bun_G(\mathcal K)$ as follows: let $\mathcal H$ and $\mathcal H'$ be objects in $Bun_G(\mathcal K)$, then the tensor product of Hilbert spaces $\mathcal H\otimes \mathcal H'$ is a $G$-equivariant $K$-bundle with $K$-grading $(\mathcal H\otimes \mathcal H')_k =\bigoplus_{x,y \in K: xy = k}\mathcal H_x\otimes \mathcal H'_y$ and twisted $G$-action
$$\sigma\rhd (h_x\otimes h_y'):= \beta[\sigma||x|y] (\sigma\rhd h_x\otimes \sigma \rhd h_y'),$$
for all $k\in K, \sigma \in G,$ $h_x \in \mathcal H_x$ and $ h_y' \in \mathcal H_y'$.




Now, for $\mathcal H$, $\mathcal H'$ and $\mathcal H''$ objects
in $Bun_G(\mathcal K)$ the associativity constraint  $$\Theta:
(\mathcal H\otimes \mathcal H')\otimes \mathcal H''\to \mathcal
H\otimes (\mathcal H'\otimes \mathcal H''),$$ for the monoidal
structure $\otimes$ is defined by $$\Theta((h_x\otimes
h_y')\otimes h_z'')= \theta[x|y|z]^{-1}h_x\otimes (h_y'\otimes
h_z'')$$ for all $x, ,y, z\in K$, $h_x\in \mathcal H_x$, $h_y'\in
\mathcal H_y'$ and $h_z''\in \mathcal H_z''$.




Finally we define the unit object $\underline{\complex}$ as the
one dimensional Hilbert space $\mathbb C$ graded only at the unit
element $e\in K$, endowed with trivial $G$-action.  All in all, we have 
that

\begin{proposition}
For  $\mathcal{K}=(K,m,1,\alpha,\beta,\theta)$ a pseudomonoid with
strict unit in the 2-category of $G$-sets with twists, the triple
$(Bun_G(\mathcal K), \otimes, \underline{\complex})$, endowed with
the tensor product $\otimes$ and the unit element
$\underline{\complex}$ is a monoidal category (or tensor
category).
\end{proposition}


\begin{rem}
Consider  the case of a  group  $G$ acting on itself by conjugation and the pseudomonoid $\mathcal{K}_G^w=(G,m,1,\alpha_w, \beta_w, \theta_w)$  defined in section \ref{subsection:definition_of_alpha_beta_theta_from_w}. In the case that $G$ is finite,  the tensor category $Bun_G(\mathcal{K}_G^w)$ is exactly the category of representations ${\rm{Rep}}(D^w(G))$ of the Twisted Drinfeld Double $D^w(G)$. Note that the quasi-Hopf algebra $D^w(G)$ is defined only for $G$ finite, but $Bun_G(\mathcal{K}_G^w)$ is defined for an arbitrary discrete group. So, the tensor category $Bun_G(\mathcal{K}_G^w)$ is a generalization of the Twisted Drinfeld Double of a finite group.


\end{rem}

\subsection{Morphism of pseudomonoids, monoidal functors and natural isomorphisms}

A morphism in the 2-category of pseudomonoids in the 2-category
$G$-sets with twists induce a monoidal functor between the
associated monoidal categories.
\begin{proposition} \label{proposition morphism pseudomonoids
induce monoidal functor} Let $\mathcal{F} =(F, \chi,
\kappa):\mathcal K \to \mathcal K'$ be a morphism of
pseudomonoids. Then $\mathcal{F}$ induces a monoidal functor from
the monoidal categories $(Bun_G(\mathcal K), \otimes ,
\underline{\complex})$ and $(Bun_G(\mathcal K'), \otimes',
\underline{\complex}')$.
\end{proposition}
\begin{proof}
Let $\mathcal{F} =(F, \chi, \kappa):\mathcal K \to \mathcal K'$ be
a morphism of pseudomonoids as defined in Proposition
\ref{proposition morphism of pseudomonoids with strict unit}. We
define a functor $\textbf{F}:Bun_G(\mathcal K)\to Bun_G(\mathcal
K')$ in the following way: for $\mathcal H$ an object in
$Bun_G(\mathcal K)$, the $K'$-graded Hilbert space
$\textbf{F}(\mathcal H)$ is the direct sum
$$\textbf{F}(\mathcal H)_y= \bigoplus_{\{x\in K: F(x)=y\}}\mathcal H_x.$$

The twisted $G$-action on $\textbf{F}(\mathcal H)$ is defined as
follows: take $h'_y \in \textbf{F}(\mathcal H)_y$ defined by the
element $h'_y = h_x$ for some vector $h_x \in \HH_x$ with
$F(x)=y$. Define the twisted $G$-action $\rhd'$ on $h'_y$ by
$$\sigma \rhd' h'_y : = \chi[\sigma|| x] (\sigma \rhd h_x).$$

The monoidal structure of the functor $\textbf{F}$ is defined as follows: take $\HH,
\HH'$ two objects in $Bun_G(\KK)$ and consider the map
\begin{eqnarray*}
  R:\textbf{F}(\mathcal H) \otimes' \textbf{F}(\HH') & \to & \textbf{F}(\HH \otimes
\HH')\\
h_{x_1} \otimes' h_{x_2} & \mapsto & \kappa(x_1,x_2)^{-1} h_{x_1}
\otimes h_{x_2}
\end{eqnarray*}
where $h_{x_1} \in \HH_{x_1}$ and $h_{x_2} \in \HH'_{x_2}$ but we
see them both as elements in $\textbf{F}(\HH)_{F(x_1)}$ and
$\textbf{F}(\HH')_{F(x_2)}$ respectively, and the element $h_{x_1}
\otimes h_{x_2}$ we see it as an element in $\textbf{F}(\HH
\otimes \HH')_{F(x_1x_2)}$. 



\end{proof}

\begin{proposition} \label{proposition 2-morphism of morphism of
pseudomonoids induce monoidal natural isomorphism}
 Given morphisms of pseudomonoids with strict unit in the 2-category of
$G$-sets with twists $\mathcal{F} =(F, \chi, \kappa)$ and
$\mathcal{F}' =(F, \chi', \kappa')$, with $\mathcal{F},
\mathcal{F}': \mathcal{K} \to \mathcal{K'}$, a 2-morphism
$\gamma:\mathcal{F} \Rightarrow \mathcal{F}'$, induces a monoidal
natural isomorphism between the monoidal functors $\rm \bf{F}$ and
$\rm \bf{F}'$.
\end{proposition}

\begin{proof}
Using the notation defined above, we define the transformation
between $\bf F$ and $\bf{F}'$ as follows:
\begin{eqnarray*}
  \textbf{F}(\HH) & \to & \textbf{F}'(\HH)\\
  h_x & \mapsto & \gamma(x)^{-1} \ h_x.
\end{eqnarray*}
Equations $\delta_1 \gamma= \chi'/\chi$ and $\delta_2 \gamma=
\kappa/\kappa'$ shown in the proof of Proposition \ref{proposition
2-morphism of pseudomonoids} imply that the transformation is
natural and monoidal, respectively.

\end{proof}

\subsection{Automorphisms of pseudomonoids and their action on the monoidal category of equivariant vector
bundles} Let us fix  $\mathcal{K}=(K,m,1,\alpha,\beta,\theta)$ a
pseudomonoid with strict unit in the 2-category of $G$-sets with
twists and let $(Bun_G(\mathcal K), \otimes,
\underline{\complex})$ be the monoidal category of $G$-equivariant
bundles over $K$.

Take $\mathcal{F} =(F, \chi, \kappa):\mathcal K \to \mathcal K$ an
invertible morphism of the pseudomonoid $\mathcal K$ and note that
Proposition \ref{proposition morphism pseudomonoids induce
monoidal functor} tells us that the induced monoidal functor
$$\textbf{F}:(Bun_G(\mathcal
K), \otimes, \underline{\complex}) \to (Bun_G(\mathcal K),
\otimes, \underline{\complex})$$ becomes an automorphism of the
monoidal category $(Bun_G(\mathcal K), \otimes,
\underline{\complex})$.

If we denote by $${\rm{Aut}}_{\rm Psmnd}(\KK)$$ the 2-group of
automorphisms of the pseudomonoid $\KK$, whose morphisms are
invertible morphisms $\FF : \KK \to \KK$ and whose 2-morphisms are
the natural transformations between functors $\gamma: \FF
\Rightarrow \FF'$, and
$${\rm{Aut}}_\otimes(Bun_G(\mathcal K))$$ the 2-group of
automorphisms of the monoidal category $(Bun_G(\mathcal K),
\otimes, \underline{\complex})$, whose morphisms are invertible
monoidal functors and whose 2-morphisms are monoidal natural
transformations, then we have that Propositions \ref{proposition
morphism pseudomonoids induce monoidal functor} and
\ref{proposition 2-morphism of morphism of pseudomonoids induce
monoidal natural isomorphism} imply that there is a 2-functor
\begin{align*}
{\rm{Aut}}_{\rm Psmnd}(\KK) & \to {\rm{Aut}}_\otimes(Bun_G(\mathcal K))\\
\gamma : \FF \Rightarrow \FF' & \mapsto {{\gamma}} :\textbf{F}
\Rightarrow \textbf{F}'
\end{align*}
from the 2-group of automorphisms of the pseudomonoid $\KK$ to the
2-group of automorpshisms of $Bun_G(\mathcal K)$.

To understand the previous action in more detail, let us start by
studying the category ${\rm{Aut}}_{\rm Psmnd}(\KK)$.

 An automorphism $\mathcal{F} =(F, \chi, \kappa):\mathcal K
\to \mathcal K$ consists of a $G$-equivariant automorphism $F\in {\rm{Aut}}_G(K)$, together with a degree 2 cochain $\chi
\oplus \kappa^{-1}$ in ${\rm Tot}^*(A^{*,*}(K \rtimes G, \Tt))$ such that
$$(\delta_G \oplus \delta_K^{(-1)^{p}})( \chi \oplus \kappa^{-1})
 = F^* \alpha / \alpha \oplus F^* \beta / \beta \oplus F^* \theta / \theta.$$
 
The automorphism $F$ lies on the image of the forgetful functor
\begin{align*}
  {\rm{Aut}}_{\rm Psmnd}(\KK) & \to {\rm Aut}_G(K) \\
\mathcal{F} =(F, \chi, \kappa) & \mapsto F;
\end{align*} if
and only if the cohomology classes $[\alpha \oplus \beta \oplus
\theta]$ and $F^*[\alpha \oplus \beta \oplus \theta]$ are equal as
cohomology classes in $H^2\left({\rm Tot}^*(A^{*,*}(K \rtimes G, \Tt))\right)$. If we
define
$${\rm Aut}_G(K;[\alpha \oplus \beta \oplus \theta]):=
 \{ F \in {\rm Aut}_G(K) | F^*[\alpha \oplus \beta \oplus \theta]=[\alpha \oplus \beta \oplus \theta]\} $$
we have that the 2-group of automorphisms of $\KK$ sits in the
exact sequence
$$\xymatrix{ 0 \ar[d]&&&\\
{\rm Tot}^1(A^{*,*}(K \rtimes G, \Tt))) \ar[d]^{\delta_G
\oplus \delta_K} && &\\
 Z^2 ({\rm Tot}^*(A^{*,*}(K \rtimes G, \Tt)))  \ar[r] & {\rm{Aut}}_{\rm Psmnd}(\KK)
\ar[r] & {\rm Aut}_G(K;[\alpha \oplus \beta \oplus \theta]) \ar[r]
& 0}$$ where $Z^2 ({\rm Tot}^*(A^{*,*}(K \rtimes G, \Tt)))$ denotes the degree
2-cocycles  and ${\rm Tot}^1(A^{*,*}(K \rtimes G, \Tt))$
parameterizes the 2-morphisms between the morphisms of $Z^2 ({\rm Tot}^*(A^{*,*}(K \rtimes G, \Tt)))  $.

If we take equivalence classes of automorphisms in
${\rm{Aut}}_{\rm Psmnd}(\KK)$ defined by the 2-morphisms, we
obtain a group which is usually denoted by
$$\pi_1 \left( {\rm{Aut}}_{\rm Psmnd}(\KK) \right);$$ this group sits in
the middle of the short exact sequence
$$
0  \to H^2 ({\rm Tot}^*(A^{*,*}(K \rtimes G, \Tt) ))\to \pi_1 \left( {\rm{Aut}}_{\rm
Psmnd}(\KK) \right) \to  {\rm Aut}_G(K;[\alpha \oplus \beta \oplus
\theta]) \to 0;$$
and by the Lyndon-Hochschild-Serre spectral sequence we know that there is an exact sequence
$$0 \to H^1(G, {\rm Hom}(K,
\Tt)) \to H^2 ({\rm Tot}^*(A^{*,*}(K \rtimes G, \Tt) ))
  \to H^2(K, \Tt)^G \stackrel{d_2}{\to} H^2(G, {\rm Hom}(K,\Tt)),$$
where $d_2$ is the differential of the second page.

 Furthermore, if we take the group of 2-morphisms of the identity morphism in
  ${\rm{Aut}}_{\rm Psmnd}(\KK)$, we obtain a group which is
  usually denoted by
$$\pi_2 \left( {\rm{Aut}}_{\rm Psmnd}(\KK) \right)$$
and is equal to $H^1( {\rm{Tot}}(A^{*,*}(K\rtimes G,\Tt)  ) )={\rm Hom}(K,
\Tt)^G$.

\subsection{The Grothendieck ring associated to the monoidal category}

Consider the isomorphism classes of objects in the monoidal
category $Bun_G(\mathcal K)$. Since the objects could be understood as finite dimensional vector spaces 
which are $K$-graded endowed with a projective $G$-action, we can add them up and
moreover we can multiply them by using the tensor product of the
monoidal category. What we obtain is a semi-ring which we can make
into a ring by applying the standard Grothendieck construction
argument of K-theory. Denoting by ${\rm{Groth}}\left(
Bun_G(\mathcal K)\right)$ the Grothendieck ring constructed from
the monoidal category $Bun_G(\mathcal K)$, we have a functor
\begin{align*} {\rm{Psmnd}}_G &\to  {\rm{Rings}} \\
{\mathcal{K}}& \mapsto {\rm{Groth}}\left( Bun_G(\mathcal K)\right)
\end{align*}
from the  2-category of pseudomonoids with strict unit in the
2-category $G$-sets with twists, to the category of rings.

The ring ${\rm{Groth}}\left( Bun_G(\mathcal K)\right)$ can also be
understood as the $\alpha$-twisted $G$-equivariant K-theory of the
monoid $K$ where the multiplication is induced by the pushforward
$m_*$ of the multiplication $m: K \times K \to K$. This twisted
K-theory ring was the main motivation of this work and is the
subject of the next section.

In the case in which $G=K$ and $G$ acts on $G$ by the left
adjoint action we have seen that any 3-cocycle $w \in Z^3(G;\Tt)$
induces a 3-cocycle $\alpha_w \oplus \beta_w \oplus \theta_w \in
Z^3({\rm{Tot}}^*(A^{*,*}(G \rtimes G, \Tt)))$ that makes $\mathcal K :=(G,m,1,\alpha_w,
\beta_w, \theta_w )$ into a pseudomonoid with strict unit in the
2-category $G$-sets with twists. In this case the Grothendieck ring
${\rm{Groth}}\left( Bun_G(\mathcal K)\right)$ is isomorphic to the
Grothendieck ring  of representations $${\rm{Groth}}({{\rm{Rep}}(D^w(G))})$$ of
the Twisted Drinfeld Double $D^w(G)$ of the group $G$ (see
\cite[Section 3]{Willerton}), which is also isomorphic to the $w$-twisted
stringy K-theory $${}^wK_{st}([G/G])$$ of the groupoid $[G/G]$
\cite[Prop. 18]{Willerton} c.f. \cite{AdemRuanZhang, BecerraUribe,
HuWang}.

\subsubsection{Automorphisms} Since the 2-functor
$${\rm{Aut}}_{\rm Psmnd}(\KK)  \to {\rm{Aut}}_\otimes(Bun_G(\mathcal
K))$$ induces a homomorphism
$$\pi_1({\rm{Aut}}_{\rm Psmnd}(\KK))  \to \pi_1({\rm{Aut}}_\otimes(Bun_G(\mathcal
K))),$$ we have that there is a homomorphism of groups
$$\pi_1({\rm{Aut}}_{\rm Psmnd}(\KK))  \to {\rm{Aut}}({\rm{Groth}}(Bun_G(\mathcal
K)))$$ which composed with the inclusion
$$H^2({\rm{Tot}}(A^{*,*}(K\rtimes G,\Tt))) \to \pi_1({\rm{Aut}}_{\rm
Psmnd}(\KK))$$ defines a homomorphism
\begin{align*} 
 H^2({\rm{Tot}}(A^{*,*}(K\rtimes G,\Tt))) \to
{\rm{Aut}}({\rm{Groth}}(Bun_G(\mathcal K))).\end{align*}

The previous morphism will be of interest when we compare it with
the group of automorphisms of the twisted equivariant K-theory
ring in section \ref{subsubsection:relation_between_automorphisms}.

\section{The fusion product and the twisted $G$-equivariant K-theory ring}\label{Section:twisted $G$-equivariant K-theory ring}

Whenever $X$ is a finite $G$-CW complex with $G$ a finite group,
the elements in $H^3_G(X;\integer)$ classify the isomorphism
classes of projective unitary stable and equivariant bundles over
$X$, and these bundles provide the required information to define
equivariant Fredholm bundles over $X$; the
homotopy groups of the space of section of a such bundle is one
way to define the twisted $G$-equivariant K-theory groups of $X$
(see \cite{BarcenasEspinozaUribeJoachim}). The homotopy classes of
automorphisms of a projective unitary stable and equivariant
bundle over $X$ are in one to one correspondence with $H^2_G(X;
\integer)$ and this group acts on the twisted $G$-equivariant
K-theory groups.

Whenever the space $X$ is a discrete $G$-set, there is an
equivalent but easier way to define the twisted $G$-equivariant
K-theory groups of $X$. Let us review it.

\subsection{Twisted $G$-equivariant K-theory}

Take a normalized 2-cocycle $\alpha: G \times G \times X \to \Tt$
and define an $\alpha$-twisted $G$-vector bundle over $X$ as a
finite dimensional $X$-graded complex  vector space $E$, which can alternatively be seen as
a finite dimensional complex vector bundle $p:E \to X$ with finite support, endowed with a $G$ action such
that $p$ is $G$ equivariant, the action of $G$ on the fibers is
complex linear, and such that the composition of the action on $E$
satisfies the equation
$$g\cdot (h \cdot z)= \alpha(g,h|| p(z)) (gh \cdot z)$$
for all $z$ in $E$. Two $\alpha$-twisted $G$-vector bundles over
$X$ are isomorphic if there exists a $G$ equivariant map $E \to
E'$ of complex vector bundles which is an isomorphism of vector
bundles.

\begin{definition}
  The $\alpha$-twisted $G$-equivariant K-theory of $X$ is the
  Grothendieck group $$\KU_G(X; \alpha)$$ associated to the semi-group
  of isomorphism classes of $\alpha$-twisted $G$-vector bundles over $X$
endowed with the direct sum operation.
\end{definition}

If we have a $G$-equivariant map $F:Y \to X$ then the pullback of
bundles induces a group homomorphism
$$F^*: \KU_G(X; \alpha) \to \KU_G(Y;F^*\alpha).$$

\subsubsection{} \label{subsubsection isomorphic twisted K-theory groups} For a normalized cochain $\chi \in C_G^1(X; \Tt)$
with $\delta_G\chi = \alpha'/\alpha$ then there is an induced
isomorphism of groups
\begin{align*}
 \overline{\chi}: \KU_G(X; \alpha) \stackrel{\cong}{\to} \KU_G(X; \alpha')
\end{align*}
where $\overline{\chi}(E):=E$ and the $G$-action $\cdot'$ on $z\in
\overline{\chi}(E)$ is given by the equation:
$$h \cdot' z:= \chi[h||p(z)] (h\cdot z).$$

Since cohomologous twistings induce isomorphic twisted K-theory
groups, we have that $H^2_G(X;\Tt)$ classifies the isomorphism
classes of twistings for the $G$-equivariant K-theory of $X$. And
since the isomorphisms $\overline{\chi}$ and
$\overline{\chi\cdot(\delta_G \gamma)}$ are equal, we have that
the group $H^1_G(X; \Tt)$ acts on the $\alpha$-twisted
$G$-equivariant K-theory group $\KU_G(X; \alpha)$ by
automorphisms.
\subsection{Pushforward}

For a $G$-equivariant map $F:Y \to X$ and $\alpha \in
Z^2_G(X;\Tt)$ there is a pushforward map
$$F_*: \KU_G(Y; F^*\alpha) \to \KU_G(X; \alpha)$$
defined at the level of vector bundles as follows
$$(F_* E)_x:= \bigoplus_{\{y \in Y | F(y)=x\}} E_y$$
where the $G$-action on $F_*E$ is the one induced by the $G$-action on $Y$ and the $G$-action on $E$.

\subsection{External product}

If we consider two $G$-sets with twist $(X,\alpha_X)$ and $(Y,
\alpha_Y)$, the external product is the homomorphism
\begin{align*}
\KU_G(X; \alpha_X) \times \KU_G(Y; \alpha_Y)
\stackrel{\boxtimes}{\to} \KU_G(X \times Y; \pi_1^*\alpha_X \cdot
\pi_2^*\alpha_Y)
\end{align*}
where $(E \boxtimes F)_{(x,y)} := E_x \otimes F_y$ and $\pi_1$ and
$\pi_2$ denote the projections of $X \times Y$ on the first and
the second coordinate respectively.
\subsection{Multiplicative structures on Twisted Equivariant K-theory}\label{Subsection:Multiplicative structures on Twisted Equivariant K-theory}
In the particular case on which the $G$-set $X$ is endowed with
the additional structure of a $G$-equivariant multiplication map
$$m: X \times X \to X$$
and moreover that the cohomology class $[\alpha]$ of the twisting
is multiplicative i.e. $\pi_1^*[\alpha] \cdot
\pi_2^*[\alpha]=m^*[\alpha]$, then the $\alpha$-twisted
$G$-equivariant K-theory group can be endowed with a product
structure. This construction could be done for $G$-equivariant
H-spaces, but for clarity we will restrict ourselves to the case
on which the $G$-set is a $G$-equivariant monoid with unit.

Let $K$ be a $G$-equivariant discrete monoid with unit and denote
by $m:K \times K \to K$ the multiplication of the monoid. Take a
twist $\alpha \in Z_G^2(K; \Tt)$ that is multiplicative, i.e. that
there exist a cochain $\beta \in C_G^1(K \times K; \Tt)$ such that
$$\delta_G\beta= \frac{m^*\alpha}{\pi_1^*\alpha \cdot
\pi_2^*\alpha}$$or equivalently $\delta_G\beta \cdot \delta_K
\alpha=1$, then we can compose the following morphisms
\begin{align*}
  \KU_G(K; \alpha) \times \KU_G(K; \alpha)
\stackrel{\boxtimes}{\To} \KU_G(K \times K; \pi_1^*\alpha \cdot
\pi_2^*\alpha) \stackrel{\overline{\beta}}{\To} & \\
\KU_G(K \times K; & m^*\alpha) \stackrel{m_*}{\To}\KU_G(K; \alpha)
\end{align*}
thus producing a product structure
\begin{align*}
 \star_\beta: \KU_G(K; \alpha) \times \KU_G(K; \alpha) & \to \KU_G(K;
 \alpha)\\
 (E,F) & \mapsto m_*(\overline{\beta}(\boxtimes(E,F))).
\end{align*}

It is a simple calculation to see that the product $\star$
previously defined is associative whenever the cohomology class
$[\beta] \in H^1_G(K \times K; \Tt)$ satisfies the equation
$$\delta_K[\beta]=1$$
as a cohomology class in $H^1_G(K \times K \times K; \Tt)$. We
therefore have that if there exists a cochain $\theta \in C^0_G(K
\times K \times K ; \Tt)$ such that
\[\delta_G \theta= \delta_K \beta\]
then the product $\star$ previously defined endows the group
$\KU_G(K; \alpha)$ with a ring structure. Summarizing:

\begin{proposition}
  Consider $\alpha \in Z_G^2(K;\Tt)$ and $\beta \in C^{1,2}(K \rtimes G; \Tt)$
  satisfying the equations $$\delta_G \beta \cdot \delta_K\alpha=1,
  \ \ \ \ \delta_2[\beta]=1.$$
Then the group $\KU_G(K;\alpha)$ endowed with the product
structure $\star_\beta$ becomes a ring. Let us denote this ring by
$$\KU_G(K; \alpha,\beta):= (\KU_G(K;\alpha), \star_\beta)$$
and let us call the pair $(\alpha,\beta)$ a multiplicative
structure for $K$.
\end{proposition}

\subsubsection{} Many of the features of the twisted
$G$-equivariant K-theory rings are better understood if we work
with the notation introduced in section \ref{Section_preliminaries}.

Recall that the double complex $A^{*,*}:=A^{*,*}(K \rtimes G ; \Tt)$ is the subcomplex of
$C^{*,*}(K \rtimes G ; \Tt)$ disregarding the 0-th row.
Consider the subcomplex $A^{*,*>3}$ of $A^{*,*}$ defined as
subcomplex of $C^{*,*}(K \rtimes G ; \Tt)$ where we disregard the
first four rows.

The double complex $A^{*,*}/A^{*,*>3}$ consists of the second, third
and fourth rows of the double complex $C^{*,*}(K \rtimes G ;
\Tt)$, and we have that if $\KU_G(K;\alpha)$ can be made into a
ring is because there exists $\beta$ and $\theta$ such that the
cochain $\alpha \oplus \beta \oplus \theta$ becomes a 3-cocycle in
the complex ${\rm{Tot}}^*(A^{*,*}/A^{*,*>3})$ and the 3-cocycle can be
seen in a diagram as follows

\begin{tikzpicture}
\matrix [matrix of math nodes,row sep=5mm]
{
 4 &  [5mm]  |(a)|  & [5mm]   & [5mm]  & [5mm] & [5mm] \\
3 & |(b)| \theta & |(c)| 1  &  & & \\
2&   & |(d)| \beta & |(e)| 1&  & \\
1& &  & |(f)| \alpha & |(g)| 1 \\
& 0 & 1 & 2 & 3&\\
};
\tikzstyle{every node}=[midway,auto,font=\scriptsize]
\draw[thick] (-1.8,-1.7) -- (-1.8,2.2) ;
\draw[thick] (-1.8,-1.7) -- (2.2,-1.7) ;
\draw[-stealth] (b) -- node {$\delta_G$} (c);
\draw[-stealth] (d) -- node {$(\delta_K)^{-1}$} (c);
\draw[-stealth] (d) -- node {$\delta_G$} (e);
\draw[-stealth] (f) -- node {$\delta_K$} (e);
\draw[-stealth] (f) -- node {$\delta_G$} (g);
\end{tikzpicture}

\begin{proposition}
  If the cocycle $\alpha \in Z^2_G(K; \Tt)$ can be lifted to a
  3-cocycle $\alpha \oplus \beta \oplus \theta$ in the complex ${\rm{Tot}}^*(A^{*,*}/A^{*,*>3})$ then the
  the group $\KU_G(K; \alpha)$ can be endowed with the ring
  structure $\KU_G(K; \alpha, \beta)$.
\end{proposition}

\subsubsection{} \label{subsubsection spectral sequence} 
Let us see another way to understand the conditions
under which the twist $\alpha$ can define a multiplicative
structure on $\KU_G(K; \alpha)$.  Consider the filtration of the
double complex $A^{*,*}$ given by the subcomplexes $F^r:= A^{*,*\geq
r}$. The spectral sequence that the filtration defines abuts to
the cohomology of the total complex of $A^{*,*}$
$$ E_\infty^{*,*} \Rightarrow H^*({\rm{Tot}}(A^{*,*}))$$ and has for first page
$$E_1^{p,q} = H^p_G(K^q; \Tt)$$ with differential $$d_1: E_1^{p,q}
\to E_1^{p,q+1}, \ \ \ d_1[x] := [(\delta_K)^{(-1)^p} x].$$

If we have a twist $\alpha$, its cohomology class $[\alpha]$ is an
element in $E_1^{2,1}$. The element $d_1[\alpha]$ is the first
obstruction to lift $\alpha$ to a 3-cocycle in
${\rm{Tot}}^*(A^{*,*})$, that is, $d_1[\alpha]=1$ if and only if
there exists $\beta \in C^{1,2}(K \rtimes G; \Tt)$ such that
$\delta_G \beta \cdot \delta_K\alpha = 1$. Note furthermore that
$$d_1[\alpha] =  \frac{m^*[\alpha]}{\pi_1^*[\alpha] \cdot
\pi_2^*[\alpha]}$$ and therefore we recover what we knew, namely
that the twist $\alpha$ may induce a product in $\KU_G(K;\alpha)$
if and only if the cohomology class $[\alpha]$ is multiplicative.

If the cohomology class $[\alpha]$ is multiplicative, then
$[\alpha]$ survives to the second page of the spectral sequence
with $[\alpha] \in E_2^{2,1}$. The second differential applied to
$[\alpha]$ is $d_2[\alpha] = [ (\delta_2 \beta)^{-1}]$

\begin{tikzpicture}
\matrix [matrix of math nodes,row sep=5mm]
{3 &  [5mm]   |(a)|  &  [5mm]  |(b)|  (\delta_K\beta)^{-1}   &   [5mm]   \\
2&    & |(c)|  \beta& |(d)| 1   \\
1& &   & |(e)|  \alpha  \\
& 0 & 1 & 2 \\
};
\tikzstyle{every node}=[midway,auto,font=\scriptsize]
\draw[thick] (-1.7,-1.3) -- (-1.7,2.0) ;
\draw[thick] (-1.7,-1.3) -- (2.0,-1.3) ;
\draw[-stealth] (e) -- node {$\delta_K$} (d);
\draw[-stealth] (c) -- node {$\delta_G$} (d);
\draw[-stealth] (c) -- node {$(\delta_K)^{-1}$} (b);
\end{tikzpicture}
\begin{tikzpicture}
\matrix [matrix of math nodes,row sep=5mm]
{3 &  [5mm]   |(a)|  &  [5mm]  |(b)|  [(\delta_K\beta)^{-1}]   &   [5mm]   \\
2&    & |(c)|  & |(d)|    \\
1& &   & |(e)|  [\alpha]  \\
& 0 & 1 & 2 \\
};
\tikzstyle{every node}=[midway,auto,font=\scriptsize]
\draw[thick] (-1.7,-1.3) -- (-1.7,2.0) ;
\draw[thick] (-1.7,-1.3) -- (2.0,-1.3) ;
\draw[-stealth] (e) -- node {$d_2$} (b);
\end{tikzpicture}
 
\noindent and this is
the second obstruction to lift $\alpha$ to a 3-cocycle in
${\rm{Tot}}^*(A^{*,*})$, i.e. $d_2[\alpha] =1$ if and only if there
exists $\theta \in C^{0,3}(K \rtimes G , \Tt)$ such that $\delta_G
\theta (\delta_K\beta)^{-1}=1$. Note that $d_2[\alpha]$ measures
the obstruction for the multiplication $\star_\beta$ in
$\KU_G(K;\alpha)$ to be associative.

We have then that the equations $d_1[\alpha]=1=d_2[\alpha]$ are
the equations that need to be satisfied in order for the group
$\KU_G(K;\alpha)$ to become a ring with respect to the
construction provided in this section.

If $[\alpha]$ survives to the third page we have that $d_3[\alpha]= [\delta_K
\theta]$ and therefore $d_3[\alpha]=1$ implies that $\alpha \oplus \beta
\oplus \theta$ is a 3-cocycle in ${\rm{Tot}}^*(A^{*,*})$. So if
$[\alpha]$ survives to the fourth page, and hence the page at
infinity, then $\mathcal{K}=(K,m,1,\alpha,\beta,\theta)$ is a
pseudomonoid with strict unit in the 2-category of $G$-sets with
twists. Summarizing:

\begin{proposition} \label{proposition obstruction for alpha to be
multiplicative}
  Consider $\alpha \in Z_G^2(K;\Tt)$, then $\alpha$  can be lifted to a
  3-cocycle $\alpha \oplus \beta \oplus \theta $ in ${\rm{Tot}}^*(A^{*,*}/A^{*,*>3})$
  if and only if $d_1[\alpha]=1=d_2[\alpha]$, and this implies that $\KU_G(K;\alpha,\beta)$ becomes a ring.
   If furthermore
  $d_3[\alpha]=1$ then $\alpha \oplus \beta \oplus \theta $ is a
  3-cocycle in ${\rm{Tot}}(A^{*,*})$ and this implies that $\mathcal{K}=(K,m,1,\alpha,\beta,\theta)$ is a
pseudomonoid with strict unit in the 2-category of $G$-sets with
twists, and in this case
$$\KU_G(K;\alpha,\beta) \cong {\rm{Groth}}\left( Bun_G(\mathcal
K)\right).$$
\end{proposition}

\subsection{Isomorphism classes of Multiplicative structures}

In this section we want to determine some sufficient conditions
under which the twisted $G$-equivariant K-theory rings $\KU_G(K;
\alpha,\beta)$ and $\KU_G(K; \alpha',\beta')$ induced by the
multiplicative structures $(\alpha, \beta)$ and $(\alpha',\beta')$
become isomorphic.

Let us consider the double complex
$A^{*,*}/{A}^{*,*>2}$ which consists of the first and the second row of
the double complex ${C}^{*,*}(K \rtimes G, \Tt)$. We claim

\begin{lemma} \label{lemma isomorphic multiplicative structures}
  Consider $(\alpha, \beta)$ and $(\alpha',\beta')$ multiplicative
  structures on $K$. If $\alpha \oplus \beta$ and $\alpha' \oplus \beta'$
  are cohomologous as 3-cocycles in ${\rm{Tot}}^*(A^{*,*}/A^{*,*>2})$,
  then the rings $\KU_G(K; \alpha,\beta)$ and $\KU_G(K;
  \alpha',\beta')$ are isomorphic.
\end{lemma}

\begin{proof}
If $\alpha \oplus \beta$ and $\alpha' \oplus \beta'$
  are cohomologous as 3-cocycles in ${\rm{Tot}}^*(A^{*,*}/A^{*,*>2})$,
  then there exists a 2-cochain  $\chi \oplus \kappa^{-1}$ with $\chi
  \in A^{1,1}$ and $\kappa \in A^{0,2}$ such that
  $$\delta_G \oplus  \delta_K^{{(-1)}^{p}}( \chi \oplus \kappa^{-1}) = \alpha'/\alpha \oplus
  \beta'/\beta,$$
  namely that $\delta_G \chi = \alpha'/\alpha$ and $(\delta_K \chi)^{-1} (\delta_G \kappa)^{-1}= \beta'/\beta$,  or diagramatically

\begin{tikzpicture}
\matrix [matrix of math nodes,row sep=5mm]
{3 &  [5mm]   |(a)|  &  [5mm]  |(c)|    &   [5mm]   \\
2& |(b)| \kappa^{-1}  & |(c)|  \beta' / \beta& |(e)|    \\
1& &  |(d)| \chi & |(e)|  \alpha' / \alpha  \\
& 0 & 1 & 2 \\
};
\tikzstyle{every node}=[midway,auto,font=\scriptsize]
\draw[thick] (-1.7,-1.3) -- (-1.7,2.0) ;
\draw[thick] (-1.7,-1.3) -- (2.0,-1.3) ;
\draw[-stealth] (b) -- node {$\delta_G$} (c);
\draw[-stealth] (d) -- node {$(\delta_K)^{-1}$} (c);
\draw[-stealth] (d) -- node {$\delta_G$} (e);
\end{tikzpicture}

The isomorphism $\overline{\chi}: \KU_G(K; \alpha)
\stackrel{\cong}{\to} \KU_G(K; \alpha')$ induces an isomorphism of
rings
$$\overline{\chi}: \KU_G(K; \alpha,\beta) \stackrel{\cong}{\to} \KU_G(K;
\alpha', \beta'')$$ where $\beta'':= \beta (\delta_K\chi)^{-1}$.
Since $(\delta_K \chi)^{-1} (\delta_G \kappa)^{-1}= \beta'/\beta$,
we have that $\beta''= \beta' (\delta_G \kappa)$, therefore the
isomorphism $\overline{\beta''}$ and $\overline{\beta'}$ are equal
and we obtain that $\KU_G(K; \alpha', \beta'') \cong \KU_G(K;
\alpha', \beta')$.
\end{proof}

The short exact sequence of complexes
$$0\to A^{*,*>2}/A^{*,*>3} \to A^{*,*}/A^{*,*>3} \to A^{*,*}/A^{*,*>2} \to
0$$ induces a long exact sequence in cohomology groups
\begin{align*}
  \to H^3({\rm{Tot}}^*(A^{*,*}/A^{*,*>3})) & \to &
  H^3({\rm{Tot}}^*(A^{*,*}/A^{*,*>2}))& \to &
  H^4({\rm{Tot}}^*(A^{*,*>2}/A^{*,*>3})) & \to \\
  [\alpha \oplus \beta \oplus \theta] & \mapsto & [\alpha \oplus
  \beta]& \mapsto & [(\delta_2 \beta)^{-1}] &
\end{align*}
and we see that Lemma \ref{lemma isomorphic multiplicative structures} and Proposition 
\ref{proposition obstruction for alpha to be multiplicative} imply that the subgroup
$$MS_G(K):=\{[\alpha \oplus \beta] \in H^3({\rm{Tot}}^*(A^{*,*}/A^{*,*>2})) | [\delta_K
\beta]=1 \}$$
 is precisely the group of equivalence classes of
multiplicative structures associated to the
  twisted $G$-equivariant K-theory of the monoid $K$. We define
  \begin{definition}
  The group 
  $$MS_G(K):=\{[\alpha \oplus \beta] \in H^3({\rm{Tot}}(A^{*,*}/A^{*,*>2})) | [\delta_2
\beta]=1 \}$$ 
will be called the group of {\it {multiplicative structures}} for the $G$-equivariant K-theory of the monoid $K$.
  \end{definition}
  
And therefore we have that
  
\begin{proposition}
  The elements of the group   $MS_G(K)$ are in one to one correspondence with the
  set of isomorphism classes of ring structures (in the sense of Lemma \ref{lemma isomorphic multiplicative structures}) on the
  twisted $G$-equivariant K-theory to the monoid $K$.
\end{proposition}

In particular we have that there are at most $\#(MS_G(K))$ of
different multiplicative structures in the twisted $G$-equivariant
K-theory groups of $K$.

\subsection{Automorphisms of the twisted equivariant K-theory
ring}

From section \ref{subsubsection isomorphic twisted K-theory
groups} we know that if $\chi \in C_G^1(K; \Tt)$ satisfies
$\delta_G \chi =1$ then the map $\overline{\chi}$ induces an
isomorphism of groups
$$ \overline{\chi}: \KU_G(K; \alpha) \stackrel{\cong}{\to} \KU_G(K;
\alpha).$$ Whenever $(\alpha,\beta)$ is a multiplicative
structure, it follows that the map $\overline{\chi}$ induces an
isomorphism of rings whenever the homomorphism $\overline{\delta_K
\chi}$ is the identity map, namely that $[\delta_K \chi]=1$ as a
cohomology class in $H^1_G(K\times K;\Tt)$. If we define the group
of multiplicative elements by
\begin{align*}
  H^1_G(K;\Tt)_{\rm{mult}} := \left\{ [\chi] \in H^1_G(K;\Tt) | \delta_K[\chi]=
  \pi_1^*[\chi] \cdot \pi_2^*[\chi] \cdot m^*[\chi]^{-1}=1 \right\}
\end{align*}
we have then that the group of automorphisms of the twisted
$G$-equivariant K-theory ring is equal to the multiplicative
elements in $H^1_G(K;\Tt)$, i.e.
\begin{align*}
  {\rm{Aut}}(\KU_G(K;\alpha,\beta)) = H^1_G(K;\Tt)_{mult}
\end{align*}

Using the spectral sequence of section \ref{subsubsection
spectral sequence} we see that the multiplicative terms appear in
the second page of the spectral sequence
$$H^1_G(K;\Tt)_{\rm{mult}} = E_2^{1,1}$$
since $d_1[\chi] =[\delta_K \chi]^{-1}$ is the obstruction of being multiplicative. If
furthermore a multiplicative element satisfies $d_2[\chi]=1$, then
we have that $[\chi]$ can be lifted to an element $[\chi \oplus
\kappa]$ in $H^2({\rm{Tot}}(A^{*,*}))$. This means that we have
the exact sequence \begin{align} \label{exact sequence for H2(A)}
0\to H^2(C^*(K;\Tt)^G) \to H^2({\rm{Tot}}(A^{*,*})) \to
H^1_G(K;\Tt)_{\rm{mult}} \stackrel{d_2}{\to}
H^3(C^*(K;\Tt)^G)\end{align} where the $G$-invariant cochains
come from the first page of the spectral sequence, i.e.
$E_1^{0,q}= C^q(K;\Tt)^G$, and its cohomology appears in the
second page, i.e. $E_2^{0,q}= H^q(C^*(K;\Tt)^G, \delta_K)$.

\subsection{Relation between the Grothendieck ring associated to a
monoidal category and the twisted equivariant K-theory ring}

Let us consider a $G$-equivariant monoid with unit $K$. We have
seen that to  a pseudomonoid with strict unit in the 2-category
of $G$-sets with twist $\mathcal{K}=(K,m,1,\alpha,\beta,\theta)$
over $K$ we can associate the ring ${\rm{Groth}}(Bun_G(\KK))$ of
isomorphism classes of objects in the monoidal category
$Bun_G(\KK)$. At the same time, since $\alpha\oplus \beta \oplus
\theta \in Z^3({\rm{Tot}}^*(A^{*,*}(K \rtimes G, \Tt)))$, then $(\alpha,\beta)$ is a
multiplicative structure and we get that
$${\rm{Groth}}(Bun_G(\KK)) \cong \KU_G(K;\alpha,\beta)$$
as rings. We have therefore a canonical map
\begin{align} \label{map between A and MS}
H^3({\rm{Tot}}(A^{*,*})) \to & MS_G(K)\\
[\alpha\oplus \beta \oplus
\theta] \mapsto & [\alpha\oplus \beta] \nonumber \end{align}
from the isomorphism classes of pseudomonoid structures with strict unit of $G$-sets with twist over $K$, to
the group of multiplicative structures of the $G$-equivariant twisted K-theory groups of $K$. 
Let us understand this map in more detail.

\vspace{0.5cm}

Consider the projection homomorphism between the exact sequences
of complexes
$$\xymatrix{
0\ar[r] & A^{*,*>2} \ar[d]^p \ar[r]^\iota & A^{*,*} \ar[d]^\pi \ar[r]^\phi &
A^{*,*}/A^{*,*>2} \ar[r]  \ar[d]^{\cong}& 0 \\
0\ar[r] & A^{*,*>2}/A^{*,*>3} \ar[r]^{\bar{\iota}} & A^{*,*}/A^{*,*>3} \ar[r]^{\bar{\phi}} &
A^{*,*}/A^{*,*>2} \ar[r] & 0 }$$ and the homomorphism between the
long exact sequences induced
$$\xymatrix{
  H^3({\rm{Tot}}(A^{*,*}) \ar@{>->}[d]^\pi \ar[r]^\phi &
  H^3({\rm{Tot}}(A^{*,*}/A^{*,*>2})) \ar[d]^{\cong} \ar[r]^{\nu} &
  H^4({\rm{Tot}}(A^{*,*>2}) \ar[d]^p \ar[r]^(0.7)\iota & \\
   H^3({\rm{Tot}}(A^{*,*}/A^{*,*>3})) \ar[r]^{\bar{\phi}} &
  H^3({\rm{Tot}}(A^{*,*}/A^{*,*>2})) \ar[r]^{\bar{\nu}} &
  H^4({\rm{Tot}}(A^{*,*>2}/A^{*,*>3})) \ar[r]^(0.7){\bar{\iota}} &
}$$ 
where $\nu$ and $\bar{\nu}$ are the connection homomorphisms with $\nu[\alpha \oplus \beta] = \delta_2[\beta]$.

By definition $$MS_G(K)= {\rm{ker}}(\bar{\nu})$$ and furthermore note that the map $\pi$ is injective since $H^3({\rm{Tot}}(A^{*,*>3}))=0$.
Therefore the canonical map defined in \eqref{map between A and MS} is precisely the map
$$H^3({\rm{Tot}}(A^{*,*})) \stackrel{\phi}{\To}  MS_G(K).$$ 
In complete generality it is difficult to give an explicit description of how the kernel and the cokernel of the map 
$H^3({\rm{Tot}}(A^{*,*})) \stackrel{\phi}{\To}  MS_G(K)$ looks like, but for calculations we can give the following description:

\begin{theorem} \label{theorem relation between H(A) and MS}
 The kernel of the homomorphism $H^3({\rm{Tot}}(A^{*,*})) \stackrel{\phi}{\To}  MS_G(K)$ is isomorphic to
\begin{align*}
{\rm{ker}}\left(H^3({\rm{Tot}}(A^{*,*})) \stackrel{\phi}{\To}  MS_G(K)\right) \cong \iota \left(H^3(C^*(K,\Tt)^G)\right)
\end{align*}
and the cokernel is isomorphic to 
\begin{align*}
{\rm{coker}}\left(H^3({\rm{Tot}}(A^{*,*})) \stackrel{\phi}{\To}  MS_G(K)\right) \cong {\rm{ker}}  \left( H^4(C^*(K,\Tt)^G) \to
H^4({\rm{Tot}}(A^{*,*})) \right).
\end{align*}
\end{theorem}
\begin{proof} Let us start with the kernel of the map $\phi$.
From the long exact sequences of cohomologies defined above we get that
\begin{align*}
{\rm{ker}}\left(H^3({\rm{Tot}}(A^{*,*})) \stackrel{\phi}{\To}  MS_G(K)\right) \cong \iota \left( H^3({\rm{Tot}}(A^{*,*>2})) \right) 
\end{align*}
where the group
$H^3({\rm{Tot}}(A^{*,*>2}))$ consists of elements in $C^{0,3}$ that are closed under the differentials $\delta_G$ and $\delta_K$
and therefore $$H^3({\rm{Tot}}(A^{*,*>2})) = Z^3(K,\Tt)^G.$$
Now, since the elements $\delta_K( C^2(K,\Tt)^G)$ are all zero in $H^3({\rm{Tot}}(A^{*,*}))$ we have that 
$$\iota \left( Z^3(K,\Tt)^G \right) = \iota \left( H^3\left( C^*(K,\Tt)^G \right)\right)$$ and therefore
\begin{align*}
{\rm{ker}}\left(H^3({\rm{Tot}}(A^{*,*})) \stackrel{\phi}{\To}  MS_G(K)\right) \cong \iota \left(H^3(C^*(K,\Tt)^G)\right).
\end{align*}

\vspace{0.5cm}

For the cokernel we have that the long exact sequence in cohomologies defined above implies that
\begin{align*}
{\rm{coker}}\left(H^3({\rm{Tot}}(A^{*,*})) \stackrel{\phi}{\To}  MS_G(K)\right) \cong H^4({\rm{Tot}}(A^{*,*>2})) \cap {\rm{ker}}(p) \cap {\rm{ker}}(\iota).
\end{align*}
Now, the projection map $$p : H^4({\rm{Tot}}(A^{*,*>2})) \to H^4({\rm{Tot}}(A^{*,*>2}/A^{*,*>3}))= H^1_G(K^3,\Tt)$$ has for kernel the elements 
 in the fourth cohomology of the $G$-invariant $K$-chains 
 $${\rm{ker}}(p) = H^4(C^*(K,\Tt)^G),$$
and this follows from the spectral sequence that converges $H^*({\rm{Tot}}(A^{*,*>2}))$ associated to the filtration 
$ {\rm{Tot}}(A^{*>q,*})$. Since the natural homomorphism  
$$ H^4(C^*(K,\Tt)^G) \to  H^4({\rm{Tot}}(A^{*,*}))$$
coincides with the map $\iota$, we have the desired isomorphism
\begin{align*}
{\rm{coker}}\left(H^3({\rm{Tot}}(A^{*,*})) \stackrel{\phi}{\To}  MS_G(K)\right) \cong {\rm{ker}}  \left( H^4(C^*(K,\Tt)^G) \to
H^4({\rm{Tot}}(A^{*,*})) \right).
\end{align*}
\end{proof}

For calculation purposes let us understand the kernel and the cokernel of the map $\phi$
from the point of view of the spectral sequence of section \ref{subsubsection spectral sequence}.
\begin{proposition} \label{proposition calculation of kernel and cokernel mu}
Consider the filtration of the complex ${\rm{Tot}}(A^{*,*})$ defined by the subcomplexes ${\rm{Tot}}(A^{*,*>q})$ and
consider the spectral sequence that it defines which converges to $H^*({\rm{Tot}}(A^{*,*})$. Then we have the isomorphisms
\begin{align*}
{\rm{ker}}\left(H^3({\rm{Tot}}(A^{*,*})) \stackrel{\phi}{\To}  MS_G(K)\right) & \cong E_4^{0,3}\\
{\rm{coker}}\left(H^3({\rm{Tot}}(A^{*,*})) \stackrel{\phi}{\To}  MS_G(K)\right) & \cong d_3(E_3^{2,1}) + d_2(E_2^{1,2}).
\end{align*}
In the particular case on which the spectral sequence collapses at the second page we conclude that
\begin{align*}
0 \to H^3(C^*(K,\Tt)^G) \to H^3({\rm{Tot}}(A^{*,*})) \stackrel{\phi}{\To}  MS_G(K) \to 0.
\end{align*}
\end{proposition}
\begin{proof}
The first page of the spectral sequence associated to the filtration
${\rm{Tot}}(A^{*,q>*})$ is given by $E_1^{0,q}= C^q(K,\Tt)^G$ and $E_1^{p,q}= H^p_G(K^q,\Tt)$ whenever $q>0$, 
and therefore one has that the second page is given by
\begin{align*}
E_2^{0,q}= & H^q(C^*(K,\Tt)^G)
\end{align*}
and for $p>0$ and $q>0$ 
\begin{align*}
E_2^{p,q}= & \frac{{\rm{ker}}\left(H^p_G(K^q,\Tt) \stackrel{\delta_K}{\to}H^p_G(K^{q+1},\Tt) \right) }
{{\rm{im}}\left(H^p_G(K^{q-1},\Tt) \stackrel{\delta_K}{\to}H^p_G(K^{q},\Tt)\right)}.
\end{align*}

The map $$H^4(C^*(K,\Tt)^G) \to
H^4({\rm{Tot}}(A^{*,*}))$$
coincides with the standard map $E_2^{0,4} \to H^4({\rm{Tot}}(A^{*,*}))$ and its kernel
consists of the images of $E_2^{1,2}$ and $E_3^{2,1}$ under the differentials $d_2$ and $d_3$ respectively, since
we know that the map
$$E_4^{0,4} \hookrightarrow H^4({\rm{Tot}}(A^{*,*})).$$
is injective. Therefore we have that
$$d_2(E_2^{1,2}) \subset {\rm{ker}}  \left( H^4(C^*(K,\Tt)^G) \to H^4({\rm{Tot}}(A^{*,*})) \right)$$
and the above inclusion is an equality whenever $d_3(E_3^{2,1})=0$. In the case that  $d_3(E_3^{2,1})\neq 0$ we could
abuse of the notation and say that
$$d_3(E_3^{2,1}) + d_2(E_2^{1,2}) = {\rm{ker}}  \left( H^4(C^*(K,\Tt)^G) \to H^4({\rm{Tot}}(A^{*,*})) \right).$$

A similar argument could be used to calculate $\iota \left(H^3(C^*(K,\Tt)^G)\right)$. Since $E_2^{0,3}=H^3(C^*(K,\Tt)^G)$
we have that its image $\iota(E_2^{0,3}) \subset H^3({\rm{Tot}}(A^{*,*}))$  is equal to the image of the canonical map
$$E_2^{0,3} \to H^3({\rm{Tot}}(A^{*,*})).$$
Since the image is isomorphic to the group to which it converges, in this case $E_4^{0,4}$, then we can conclude
that $\iota \left(H^3(C^*(K,\Tt)^G)\right) \cong E_4^{0,3}$ and therefore
\begin{align*}
{\rm{ker}}\left(H^3({\rm{Tot}}(A^{*,*})) \stackrel{\phi}{\To}  MS_G(K)\right) \cong E_4^{0,3}.
\end{align*}

Finally, whenever the spectral sequence collapses at the second page we have that $d_2=0=d_3$ and therefore $\phi$ is surjective.
Since we have in this case we have that $E_4^{0,3}=E_2^{0,3} = H^3(C^*(K,\Tt)^G)$, the proposition follows.
 \end{proof}

 From Theorem \ref{theorem relation between H(A) and MS} we can deduce two things.
 \begin{itemize}[leftmargin=*]
 \item If $\KK=(K,m,1,\alpha,\beta,\theta)$ is a pseudomonoid with strict unit in the 2-category of $G$-sets with twists such that
 $[\alpha \oplus \beta \oplus \theta]$ lies in the image of $\iota$, then the Grothendieck ring ${\rm{Groth}}(Bun_G(\KK))$
 is isomorphic to the untwisted ring  $\KU_G(K)$.
 \item Multiplicative structures $(\alpha',\beta')$ in $MS_G(K) $ such that $\alpha' \oplus \beta' \oplus \theta'$
 belongs to $Z^3({\rm{Tot}}(A^{*,*}/A^{*,*>3})) $  and $\delta_K\theta' \neq 0$, 
 define ring
 structures $\KU_G(K; \alpha',\beta')$ which cannot be obtained as the Grothendieck ring ${\rm{Groth}}(Bun_G(\KK))$
 for any pseudomonoid $\KK$ with strict unit in the 2-category of $G$-sets with twists.
 \end{itemize}

\subsubsection{Relation between the automorphisms} \label{subsubsection:relation_between_automorphisms}
We have seen that the isomorphism classes of automorphisms of
$\KK$ that leave the monoid $K$ fixed is isomorphic to the group
$H^2( {\rm{Tot}}(A^{*,*}(K\rtimes G,\Tt)))$.  Since the automorphism group of
$\KU_G(K; \alpha, \beta)$ is $H^1_G(K;\Tt)_{\rm{mult}}$ we have
that there is an induced map
$$H^2( {\rm{Tot}}(A^{*,*})) \to H^1_G(K;\Tt)_{\rm{mult}}$$
which matches the homomorphism that appears in the exact
sequence \eqref{exact sequence for H2(A)}
\begin{align*} 
0\to H^2(C^*(K;\Tt)^G) \to H^2({\rm{Tot}}(A^{*,*})) \to
H^1_G(K;\Tt)_{\rm{mult}} \stackrel{d_2}{\to}
H^3(C^*(K;\Tt)^G).\end{align*}
Note that in the case that the spectral sequence collapses at the second page we get the short exact sequence
\begin{align*} 
0\to H^2(C^*(K;\Tt)^G) \to H^2({\rm{Tot}}(A^{*,*})) \to
H^1_G(K;\Tt)_{\rm{mult}} \to
0.\end{align*}

\vspace{0.5cm} A more elaborate analysis of the homomorphisms
$$H^3( {\rm{Tot}}(A^{*,*})) \stackrel{\phi}{\to} MS_G(K) \ \ \ \mbox{and} \ \ \
H^2( {\rm{Tot}}(A^{*,*})) \to H^1_G(K;\Tt)_{\rm{mult}}$$ will
depend on the choice of the group $G$ and of the $G$-equivariant
monoid $K$. In the next chapter we will calculate explicitly the
previous homomorphisms for several examples, and from them we will
deduce interesting information with regard to the twisted
equivariant K-theory rings.

\section{Examples}\label{Section: Examples}
The main objective of this section is to use Proposition \ref{proposition calculation of kernel and cokernel mu}  to calculate the
kernel and cokernel of the homomorphism
$$H^3( {\rm{Tot}}(A^{*,*})) \stackrel{\phi}{\to} MS_G(K)$$ for
different choices of $G$ and $K$,
in order to show the different
twisted $G$-equivariant K-theory rings over $K$ that can appear.

\subsection{Trivial action of $G$ on $K$} In this case we have that the spectral sequence defined in Proposition
\ref{proposition calculation of kernel and cokernel mu} collapses at the second page and moreover we have that 
$C^*(K,\Tt)=C^*(K,\Tt)^G$. Therefore we obtain the short exact sequence 
$$ 0 \to H^3(K,\Tt) \to H^3( {\rm{Tot}}(A^{*,*})) \stackrel{\phi}{\to} MS_G(K) \to 0,$$
thus implying that $$MS_G(K)\cong H^3( {\rm{Tot}}^*(B^{*,*}(K \times G,\Tt))),$$and moreover that all multiplicative structures for the $G$-equivariant K-theory of $K$ can be obtained from the ring structures
defined by the Grothendieck rings of the monoidal categories $Bun_G(\KK)$. We also obtain the short
exact sequence
$$ 0 \to H^2(K,\Tt) \to H^2( {\rm{Tot}}(A^{*,*})) \stackrel{}{\to} {\rm{Hom}}(K, {\rm{Hom}}(G, \Tt))\to 0$$
where in this case $H^1_G(K,\Tt)_{mult}={\rm{Hom}}(K, {\rm{Hom}}(G, \Tt))$.

 Furthermore, if $[\theta] \in H^3(K,\Tt)$
is non trivial then we can define a non-trivial  pseudomoinoid with strict unit in the 2-category
of $G$-sets with twist $\mathcal{K}=(K,m,1,0,0,\theta)$ with $[0 \oplus 0 \oplus \theta]$ non-zero in $H^3( {\rm{Tot}}(A^{*,*}))$, such that
$${\rm{Groth}}(Bun_G(\KK)) \cong R(G) \otimes_\integer \integer[K]$$
where $R(G)$ is the Grothendieck ring of finite dimensional complex representations of $G$ and $\integer[K]$ is the group
ring of $K$, since we know that $R(G) \otimes_\integer \integer[K]$ is isomorphic to the non-twisted ring structure on $\KU_G(K)$.

\subsubsection{$G=\integer/n$ and $K=\integer/m$} In this case we have that $H^3(K,\Tt)=\integer/m$ and 
$$H^3( {\rm{Tot}}(A^{*,*}))=\integer/m \oplus \integer/(n,m)$$ 
where $(n,m)$ is the greatest common divisor of the pair $n,m$. 
Therefore $MS_G(K)= \integer/(n,m)$ and we have that all non trivial multiplicative structures come from the group
$$\integer/(n,m)=H^3( {\rm{Tot}}(B^{*,*}))\subset H^3( {\rm{Tot}}(A^{*,*})).$$

\subsection{Adjoint action of $G$ on itself}\label{subsection: G actiing on itself}
From Lemma \ref{lem:iso_A_with_B_and_cohomology_of_K} we know that in this
case we have the split short exact sequence
$$
\xymatrix{ 0 \ar[r] & H^*( {\rm{Tot}}(B^{*,*})) \ar[r] & H^*( {\rm{Tot}}(A^{*,*})) \ar[r] &H^*(G,\Tt)) \ar@/_/ @{..>}[l]_(0.4){\mu^*} \ar[r] & 0}$$
induced by the short exact sequence of complexes $$0 \to  {\rm{Tot}}(B^{*,*}) \subset  {\rm{Tot}}(A^{*,*}) \to C^{*>0}(G,\Tt) \to 0.$$

Since we have the inclusion $H^*( {\rm{Tot}}(B^{*,*})) \subset H^*( {\rm{Tot}}(A^{*,*}))$, we can deduce that the groups $E_2^{*,0}$
of the 0-th column of the second page of the spectral sequence converging to $H^*( {\rm{Tot}}(A^{*,*}))$ defined in Proposition \ref{proposition calculation of kernel and cokernel mu},
 are unaffected by the differentials $d_i$ for $i>1$; this follows from the injectivity between  the spectral sequences associated to the filtrations
 $B^{*,*>q}$ and $A^{*,*>q}$ of ${\rm{Tot}}(B^{*,*})$ and ${\rm{Tot}}(A^{*,*})$ respectively.

 Therefore we have that $E_4^{0,3}=E_2^{0,3}=H^3(C^*(G,\Tt)^G)$ and
$d_2(E_2^{1,2})=0=d_3(E_3^{2,1})$, and by Proposition \ref{proposition calculation of kernel and cokernel mu} we get the short exact
sequence
\begin{align} \label{short_exact_sequence_of_MS}
 0 \to H^3(C^*(G,\Tt)^G) \to H^3( {\rm{Tot}}(A^{*,*})) \stackrel{\phi}{\to} MS_G(G) \to 0.
 \end{align}

Moreover, the cokernel of the inclusion 
$$E_\infty^{0,q} \to H^q( {\rm{Tot}}(A^{*,*})) $$
should match the cohomology group $H^q( {\rm{Tot}}(B^{*,*})) $ since this piece is built from the groups $E_\infty^{r,s}$ with $r,s>1$ and $r+s=p$, therefore we have the canonical isomorphism
$$ H^q(C^*(G,\Tt)^G) \oplus H^q( {\rm{Tot}}(B^{*,*})) \stackrel{\cong}{\to} H^q( {\rm{Tot}}(A^{*,*})), x\oplus y \mapsto x+y$$
which in particular implies that
$$H^q(C^*(G,\Tt)^G) \stackrel{\cong}{\To} H^q(G,\Tt).$$

 Then we can conclude that the
composition of the maps
$$H^3( {\rm{Tot}}(B^{*,*})) \subset H^3( {\rm{Tot}}(A^{*,*})) \stackrel{\phi}{\to} MS_G(G) $$
is an isomorphism, and therefore we obtain the canonical isomorphism
\begin{equation}\label{equation final remark}
H^3( {\rm{Tot}}(B^{*,*})) \stackrel{\cong}{\to}MS_G(G), \hspace{0.5cm} [\alpha \oplus \beta] \mapsto [\alpha \oplus \beta].
\end{equation}

Also we obtain the short exact sequence
\begin{align*} 
0\to H^2(C^*(G;\Tt)^G) \to H^2({\rm{Tot}}(A^{*,*})) \to
H^1_G(G;\Tt)_{\rm{mult}} \to
0,\end{align*}
with $H^1_G(G;\Tt)_{\rm{mult}} \cong H^2( {\rm{Tot}}(B^{*,*})) \cong {\rm{Hom}}(G, {\rm{Hom}}(G, \Tt))$.
  
 Now we will study the multiplicative structures that define the pseudomonoids constructed via the formalism
 introduced in \cite{Dijkgraaf}, and whose properties were outlined in section \ref{subsection:definition_of_alpha_beta_theta_from_w}.
  For this purpose we need to calculate the composition of the maps
  $$H^3( G,\Tt) \stackrel{\tau_1^*}{\to} H^3( {\rm{Tot}}(A^{*,*})) \stackrel{\phi}{\to} MS_G(G)$$
  where $\tau_1^*$ is the induced map in cohomology which was defined at the chain level in \eqref{shuffle_hom_for_A}.
  This calculation will be carried out using the ring structure of the ring $H^*(G,\integer)$ together with the pullback map
  $\mu^* : H^*(G,\integer) \to H^*(G \rtimes G,\integer)$
  induced by the multiplication $\mu: G \rtimes G \to G$ as we have in Theorem \ref{Theorem: mu bullback}. Since we have the isomorphism
  $$H^3( {\rm{Tot}}(A^{*,*})) \oplus H^3(G,\Tt) \stackrel{\cong}{\to} H^3(G \rtimes G,\Tt), \hspace{0.5cm} x \oplus y \mapsto x + \pi_2^*y$$
  where $\pi_2 : G \rtimes G \to G, (a,g) \mapsto g$ is the projection, then the short exact sequence of \eqref{short_exact_sequence_of_MS} implies that the following is also a
   short exact sequence
  $$0 \to H^3(C^*(K;\Tt)^G) \oplus H^3(G,\Tt) \to H^3(G \rtimes G,\Tt) \to  MS_G(G) \to 0.$$
  
  We conclude that 
  the homomorphism $$\phi \circ \tau_1^* : H^3(G,\Tt) \to MS_G(G)$$
  is equivalent to the composition of homomorphisms
  $$ H^3(G,\Tt) \stackrel{\mu^*}{\to} H^3(G \rtimes G,\Tt)/(H^3(C^*(K;\Tt)^G) \oplus H^3(G,\Tt)) \stackrel{\cong}{\to} MS_G(G).$$

Therefore we simply need to find the restriction to $H^3({\rm{Tot}}(B^{*,*}))$ of the image of $\mu^*$ in $ H^3(G \rtimes G,\Tt)$,
  $$ H^3(G,\Tt) \stackrel{\mu^*}{\to} H^3(G \rtimes G,\Tt) \stackrel{{\rm{res}}}{\to} H^3({\rm{Tot}}(B^{*,*})) \cong MS_G(G).$$
  Let us see some examples.

  \subsubsection{Cyclic groups}
Consider $G=K=\integer/n$ and note that as rings $H^*(G;\integer)=\integer[x]/(nx)$ where $|x|=2$. By the Kunneth theorem
we have that $$H^2(G \times G;\integer) \cong H^2(G;\integer) \otimes  H^0(G;\integer) \oplus H^0(G;\integer) \otimes  H^2(G;\integer) $$ and moreover we have that
$$ \mu^*x = x \otimes 1 + 1\otimes x.$$
Since $x^2$ is the generator of $H^4(G;\integer)$ we obtain 
$$ \mu^*x^2 = x^2 \otimes 1 + 2 x \otimes x + 1\otimes x^2$$
where
$$H^4({\rm{Tot}}(B^{*,*}(G \times G, \integer))) = \langle x \otimes x \rangle \cong \integer/n$$
and therefore we obtain that the map $$\phi \circ \tau_1^*: H^3(G,\Tt) \to MS_G(G)$$ is equivalent to the map
$$H^4(G; \integer) \to H^4({\rm Tot}(B^{*,*}(G \times G, \integer))),  \  x^2 \mapsto  2 x \otimes x.$$
Since in this case $H^3(G; \Tt)\cong MS_G(G) \cong \integer/n$ we have that $\phi \circ \tau_1^*: \integer/n \stackrel{\times 2}{\to} \integer/n$. 

Therefore, when $n$ is odd, the map $\phi \circ \tau_1^*$ is an isomorphism and therefore the Grothendieck rings of representations ${\rm{Groth}}({\rm{Rep}}(D^w(G)))$ of the Twisted Drinfeld Doubles $D^w(G)$ for $w \in H^3(G,\Tt)$ are all non-isomorphic.

Meanwhile when $n$ is even, the cocycle $w \in Z^3(G,\Tt)$ whose cohomology class is $\frac{n}{2} \in \integer/n \cong H^3(G,\Tt)$
has for Grothendieck ring of representations ${\rm{Groth}}({\rm{Rep}}(D^w(G)))$ a ring isomorphic to ${\rm{Groth}}({\rm{Rep}}(D(G)))$, which is by definition the ring $\KU_G(G)=R(G) \otimes \integer G$. Moreover, the multiplicative structures defined by odd numbers in $\integer/n \cong MS_G(G)$ 
define Grothendieck rings of representations associated to the respective pseudomonoids which cannot be recovered via the Grothendieck ring of representations associated to the Twisted Drinfeld Double construction.  

Note furthermore that in this case the automorphism groups are isomorphic
$$H^2({\rm Tot}(A^{*,*}(G \times G, \Tt))) \cong H^1_G(G,\Tt)_{mult} \cong \integer/n.$$

\subsubsection{Quaternionic group} Consider $G=K=Q_8$  the quaternionic group and recall that $Q_8 \subset SU(2)$ and that
it sits in the short exact sequence
$$0 \to \integer/2 \to Q_8 \to \integer/2 \oplus \integer/2 \to 0.$$
From the Serre fibration
$$SU(2)/Q_8 \to ESU(2)/Q_8 \to BSU(2)$$
one can deduce that the integral cohomology ring of $Q_8$ is
$$H^*(Q_8,\integer) = \integer[x,y,e]/ (x^2,y^2,xy,2x,2y,xe,ye,8e) \text{ with } |x|=|y|=2 \text{ and } |e|=4.$$

Since the cohomology of $Q_8$ is all of even degree, by the Kunneth isomorphism we have that
$$ H^4(Q_8 \times Q_8,\integer) \cong \bigoplus_{j=0}^4  H^j(Q_8,\integer) \otimes H^{4-j}(Q_8,\integer) $$
and since $Q_8 \times Q_8 \cong Q_8 \rtimes Q_8$ we have that
$$ H^*(Q_8 \rtimes Q_8,\integer) \cong H^*(Q_8 \times Q_8,\integer) .$$
In this case we have that
$$MS_{Q_8}(Q_8)\cong H^4({\rm Tot}(B^{*,*}(Q_8 \rtimes Q_8, \integer))) =\langle x\otimes x, x\otimes y, y\otimes x,y\otimes y\rangle \cong (\integer/2)^{\oplus 4}$$
and since $$\mu^*e = e \otimes 1 + 1 \otimes e$$
 we can deduce that the map
$$\phi \circ \tau_1^* : H^3(Q_8,\Tt) \stackrel{\times 0}{\to} MS_{Q_8}(Q_8)$$
is the trivial map.

A nice consequence of the triviality of the map $\phi \circ \tau_1^*$ is that  for all $w \in Z^3(Q_8,\Tt)$, the Grothendieck ring of 
representations of the $w$-Twisted Drinfeld Double is isomorphic as rings to $\KU_{Q_8}(Q_8)$, which is the Grothendieck ring of representations of the untwisted Drinfeld double $D(Q_8)$.

Note also that the automorphism groups are isomorphic
$$ H^2({\rm Tot}(A^{*,*}(Q_8 \rtimes Q_8, \Tt)))\cong H^1_{Q_8}(Q_8,\Tt)_{mult}\cong (\integer/2)^{\oplus 4}.$$ 

\subsubsection{$G$ simple and non-abelian.} In section \ref{G_simple_non_abelian} we showed that 
$H^3({\rm Tot}(B^{*,*}(G \rtimes G, \Tt)))=0$ and therefore we have that $MS_G(G)=0$. Hence all Grothendieck rings
of representations for all pseudomonoids are isomorphic to the ring $\KU_G(G)$.
\subsubsection{$G$ binary icosahedral.} The same result applies to the binary icosehedral group since we showed in
section \ref{G_binary_icosahedral} that $H^3({\rm Tot}(B^{*,*}(G \rtimes G, \Tt)))=0$ and hence $MS_G(G)=0$.

\subsection{$\integer/n$ acted by $\integer/2$} Consider the action of $G=\integer/2$ on $K=\integer/n$ given by multiplication of $-1$.
The group $\integer/n \rtimes \integer/2$ is isomorphic to the dihedral group $D_n$ of rigid symmetries of the regular polygon of $n$ sides.

\subsubsection{$n$ odd} Let us suppose that $n $ is odd and recall that in this case $H^1(D_n, \Tt)=\integer/2$, $H^2(D_n, \Tt)=0$ and $H^3(D_n, \Tt)=\integer/2 \oplus \integer/n$. Since $H^3(\integer/2,\Tt)=\integer/2$ we can conclude that $H^3({\rm{Tot}}(A^{*,*}))=\integer/n$.
Now, applying the spectral sequence defined in Proposition \ref{proposition calculation of kernel and cokernel mu}
we have that $E_1^{1,q}=H^1_{G}(K^q,\Tt) =\integer/2$, since the only fixed point of the $\integer/2$ action is the $p$-tuple of zeros,
and therefore we obtain that $E_2^{1,q}= 0$ for $q>0$ since a simple calculation shows that
 the maps $d_1: E_1^{1,2i-1} \to E_1^{1,2i}$ are all isomorphisms. Moreover, since $E_1^{2,q}= H^2_G(K^q,\Tt)=0$ because
$H^2(\integer/2,\Tt)=0$ we have that $E_3^{2,1}=0$ and we can conclude that 
$${\rm{coker}}\left(H^3({\rm{Tot}}(A^{*,*})) \stackrel{\phi}{\To}  MS_G(K)\right) =0;$$
 hence we have that the map $H^3({\rm{Tot}}(A^{*,*})) \stackrel{\phi}{\To}  MS_G(K)$ is surjective.

It remains now to calculate ${\rm{ker}}\left(H^3({\rm{Tot}}(A^{*,*})) \stackrel{\phi}{\To}  MS_G(K)\right)$. Applying the same argument
as before we have that $E_2^{0,3}= H^3(C^*(K,\Tt)^G)$  and we already know that $E_2^{1,2}=0=E_2^{2,1}$; therefore
$\iota(H^3(C^*(K,\Tt)^G))$ coincides with the image of the canonical map $E_2^{0,3} \to H^3({\rm{Tot}}(A^{*,*}))$
and this map must be surjective. Therefore we have that 
$${\rm{ker}}\left(H^3({\rm{Tot}}(A^{*,*})) \stackrel{\phi}{\To}  MS_G(K)\right)= \integer/n$$
and we can conclude that $$MS_G(K)=0;$$
i.e. all multiplicative structures on the $\integer/2$-equivariant K-theory of $\integer/n$ are trivial and
and all the Grothendieck rings $\rm{Groth}(Bun_{\integer/2}(\KK))$ are isomorphic to the ring $\KU_{\integer/2}(\integer/n)$
for any $\KK=(\integer/n,m,1,\alpha,\beta,\theta)$. By Remark \ref{Remark: Bun_G(K) 3-cocyclo trivial}, $\KU_{\integer/2}(\integer/n)$ is just the ring of isomorphism classes of representations  of the dihedral group $D_n$. 

In this case the automorphism groups are both trivial
$$ H^2({\rm Tot}(A^{*,*}(\integer/n \rtimes \integer/2, \Tt)))=0=H^1_{\integer/2}(\integer/n,\Tt)_{mult}.$$ 

\subsubsection{$n$ even} Let us now suppose that $n $ is even; in this case $H^1(D_n, \Tt)=\integer/2 \oplus \integer/2$, $H^2(D_n, \Tt)=\integer/2$ and $H^3(D_n, \Tt)=\integer/2  \oplus \integer/2 \oplus \integer/n$. Since $H^3(\integer/2,\Tt)=\integer/2$ we have that $H^3({\rm{Tot}}(A^{*,*}))=\integer/2 \oplus \integer/n$.
Now, we also have that $$E_1^{1,q}=H^1_{G}(K^q,\Tt) \cong {\rm{Maps}}((\integer/2)^q, \integer/2)$$ since the fixed points of the $\integer/2$ action on $K^q$ consists of $q$-tuples of points  with either $0$ or $\frac{n}{2}$ for entries. It is a simple calculation
to show that the differential $d_1: E_1^{1,i} \to E_1^{1,i+1}$ is precisely the differential of the cohomology of the group $\integer/2$
with coefficients in $\integer/2$ and therefore we get that
$$E_2^{1,q}\cong H^q(\integer/2, \integer/2) \cong \integer/2.$$

The groups $E_1^{2,q}$ are trivial because $H^2(\integer/2, \Tt)=0$.

Let us now calculate the group $E_2^{0,1}= H^1(C^*(K,\Tt)^G)$. This group consists of the 
maps $f: \integer/n \to \Tt$ such that $f$ is invariant under the $G$-action, namely $f(x)=f(-x)$, and
that $\delta_1f=0$, which means that $f$ is a homomorphism. The only $G$-invariant homomorphisms are the ones
that take values in the subgroup $\integer/2 \subset \Tt$ and therefore we have that
$$E_2^{0,1}= H^1(C^*(K,\Tt)^G) = \integer/2.$$

The information we have obtained so far on the cohomology groups of the second page of the spectral sequence is the following

\begin{tikzpicture}
\matrix [matrix of math nodes,row sep=5mm]
{
 3 &  [5mm]  |(a)|  ? & [5mm]   & [5mm]  & [5mm] & [5mm] \\
2 & |(b)| ? & |(c)|  \integer/2  &  & & \\
1&  \integer/2 & |(d)|  \integer/2 & |(e)| 0&  & \\
0& \Tt &  \integer/2 & |(f)| 0 & |(g)|  \integer/2 \\
& 0 & 1 & 2 & 3&\\
};
\tikzstyle{every node}=[midway,auto,font=\scriptsize]
\draw[thick] (-2.2,-1.7) -- (-2.2,2.2) ;
\draw[thick] (-2.2,-1.7) -- (2.2,-1.7) ;
\end{tikzpicture}

\noindent and since we know that 
$H^1(D_n, \Tt)=\integer/2 \oplus \integer/2$ and $H^2(D_n, \Tt)=\integer/2$ we can deduce that
$E_\infty^{0,2}=E_2^{0,2}=0$ and therefore $E_2^{0,3}=E_4{0,3}= \integer/n$. Hence we have that
$$H^3(C^*(G,\Tt)^G)= \integer/n, \hspace{0.5cm} MS_G(K)= \integer/2 \text{ and } $$ 
$$H^3({\rm{Tot}}(A^{*,*})) \stackrel{\phi}{\To}  MS_G(K), \hspace{0.5cm} \integer/n \oplus \integer/2 \to \integer/2$$ is the canonical projection map. 

In this case the automorphism groups are isomorphic 
$$ H^2({\rm Tot}(A^{*,*}(\integer/n \rtimes \integer/2, \Tt))) \cong H^1_{\integer/2}(\integer/n,\Tt)_{mult} \cong \integer/2.$$ 

\section{Appendix: Relation with (coquasi) bialgebras}\label{Appendix: Relation with (coquasi) bialgebras}



In what follows we will show how $Bun_G(\mathcal K)$ can be understood as the tensor category of corepresentations associated
to an explicit coquasi-bialgebra, and for this purpose we will show that the input necessary for defining such coquasi-bialgebra is equivalent to the information encoded in a pseudomonoid with
strict unit in the 2-category of $G$-sets with twists.

\subsection{Coquasi-bialgebras}
\subsubsection{Coalgebras and comodules}
Let $k$ be a field. A coalgebra over $k$ is a vector space over $k$
together with two linear maps  $\Delta:C\to C\ot C,$ $\va:C\to k$
(called comultiplication and counit respectively) such that
$(C\otimes\Delta) \Delta= (\Delta\ot C)\Delta$ and $(C\ot
\va)\Delta=(\va\otimes C)\Delta= C$. We shall use the Sweedler's
notation omitting the sum symbol, that is $\Delta(c)= c_1\otimes
c_2$ if $c\in C$.

If $C$ is a coalgebra, a right $C$-comodule is a vector $k$-space
$M$ with a linear map $\rho: M\to M\ot C$ such that $(\rho\ot C)\rho
= (M\ot \Delta)\rho$ and $( M\ot \va)=M$. Again for the comodule structure we
shall use Sweedler's notation omitting the sum symbol, i.e.,
$\rho(m)= m_{0}\otimes m_1,$ $m_{0}\in M, m_1\in C$. If $M, N$ are
$C$-comodules, a comodule map is a linear map $f:M\to N$ such that
$\rho_N f= (f\ot C)\rho_M$. We shall denote by $\mathcal{M}^C$ the
category of right $C$-comodules.

If $C, C'$ are coalgebras $C\otimes C'$ is a coalgebra with
comultiplication $\Delta(c\ot c')= (c_1\ot c'_1)\ot (c_2\ot c'_2)$
and counit $\va(c\ot c')=\va(c)\va(c')$.

For a coalgebra $C$ the space $C^*$ is an associative algebra with
the convolution product $f*g(c)= f(c_1)f(c_2)$ and unit $\va$.

\subsubsection{Coquasi-bialgebras}
 A coquasi-bialgebra is a
five-tuple $(H, \Delta, m, 1_H , \phi )$ where $H$ is a
coassociative coalgebra with counit, $m: H\ot H\to H,h\ot g\mapsto
hg$ is a coalgebra map, $1_H$ is a grouplike element (i.e. $\Delta(1_H)=1_H\otimes 1_H$)  which is a unit
for $m$, and
 $\phi\in (H\ot H\ot H)^{*}$ is a convolution invertible map (called the
\emph{coassociator}), satisfying the identities
\begin{eqnarray}
&&\phi(g\ot 1_H\ot h)=\va(g)\va(h),\label{q1}\\
&&  m(m\ot \id_H)*\phi= \phi* m(\id_H\ot m),\label{q2}\\
&&  \phi(d_1f_1\ot g_1\ot h_1)\phi(d_2\ot f_2\ot g_2h_2)=
\phi(d_1\ot f_1\ot g_1)\phi( d_2\ot f_2g_2\ot h_1)\phi(f_1\ot g_3\ot
h_2)\label{q3}
\end{eqnarray}
for all $f,g,h\in H$.


The  category of right $H$-comodules $\mathcal{M}^H$ is monoidal,
where the tensor product is over the base field and the comodule
structure of the tensor product is induced by the multiplication.
The associator is given by
\begin{eqnarray*} \Phi _{U,V,W} &:&(U\otimes V)\otimes
W\longrightarrow U\otimes (V\otimes W) \\ \Phi_{U,V,W}((u\otimes
v)\otimes w) &=&\phi^{-1} (u_{1},v_{1},w_{1})u_{0}\otimes
(v_{0}\otimes w_{0})
\end{eqnarray*} for $u\in U$, $v\in V$, $w\in W$ and $U,V,W\in
\mathcal{M}^H$.

\medbreak



\subsection{Coquasi-bialgebras associated to pseudomonoids}

Let $G$ be a finite  group and
$\mathcal{K}=(K,m,1,\alpha,\beta,\theta)$ a pseudomonoid with strict
unit in the 2-category of $G$-sets with twists; let us denote $\C^G:= {\rm{Maps}}(G,\C)$, let $\delta_\sigma \in \C^G$  be the function that assigns $1$ to $\sigma$ and $0$ otherwise, and let $\delta_{\sigma,\tau}$ be the Dirac's delta associated to the pair $\sigma,\tau \in G$, namely $\delta_{\sigma,\tau}$ is $1$ whenever $\sigma= \tau$ and $0$ otherwise.

\begin{theorem}\label{appen theorem 1}

The vector space $\C^G\# K$ with basis
$\{\delta_\sigma\# x| \sigma\in G, x\in K\}$ is a coquasi-bialgebra
with product
$$(\delta_\sigma\#x)(\delta_\tau \# y)=
\delta_{\sigma,\tau}\beta[\sigma||x|y]\delta_\sigma\#xy,$$ coproduct,
$$\Delta(\delta_\sigma\# x)=
\sum_{a,b\in G: ab=\sigma}\alpha[a|b||x]\delta_a\# bx\ot
\delta_b\#x,$$ associator $$\phi(\delta_\sigma\# x,\delta_\tau\#
y,\delta_\rho\#
z)=\delta_{\sigma,e}\delta_{\tau,e}\delta_{\rho,e}\theta[x|y|z],$$
counit $\varepsilon(\delta_\sigma\# x)=\delta_{\sigma,e}$, and unit
$1\#e$ for all $\sigma,\tau,\rho\in G$, $x,y,z\in K$.

Moreover, the tensor category of right $\C^G\# \C K$-comodules is tensor isomorphic to the monoidal category $Bun_G(\mathcal K)$ of equivariant vector
bundles on $\mathcal K$.
\end{theorem}

\begin{proof}
It is straightforward to check that  $\C^G\# K$ satisfies all axioms of coquasi-bialgebra.

If $V=\bigoplus_{x\in K}V_x$ is an object in $Bun_G(\mathcal K)$, we
define a right $\C^G\# \C K$-comodule structure over $V$ by
$$\rho(v_x)=\sum_{\sigma\in G}\sigma \rhd v_x\otimes \delta_\sigma\#
x,$$ for all $v_x\in V_x$, $x\in K$. It follows that this rule defines a strict monoidal functor 
$\mathcal{F}$ from $Bun_G(\mathcal K)$ to the category of right $\C^G\# \C K$-comodules.






Conversely, if $(V,\rho)$ is a right $\C^G\# \C K$-comodule then
$$V=\bigoplus_{x\in K} V_x,\  \  \ \text{where}\   V_x=\{v\in V: (id\ot \pi)\rho= v\otimes x\},$$ where
 $\pi: \C^G\# K\to \C K, \delta_{\sigma}\#
x\mapsto \delta_{\sigma, e}u_x$ and the comodule map defines unique
linear maps $\rhd: \C G\otimes V_x\to V$ $(x\in G)$ such that
$$ \rho(v_x)= \sum_{\sigma \in G}\sigma\rhd v_x\# x,$$ for all $v_x\in
V_x$. It follows that $(V=\bigoplus_{x\in K} V_x, \rhd) \in Bun_G(\mathcal K)$, and this rule defines a functor that by construction is the inverse of $\mathcal{F}$.

\end{proof}
\begin{rem}
 If $\theta$ is trivial, $\C^G\# K$ is a Hopf algebra, and if we consider
 $G$ acting by conjugation on $K=G$ together with the 3-cocycle $w \in Z^3(G,\Tt)$, the coquasi-bialgebra
$\C^G \# \C G$ defined by the pseudomonoid $(G,m,1, \alpha_w,\beta_w,\theta_w)$  as in section \ref{subsection:definition_of_alpha_beta_theta_from_w}, is the dual of the Twisted Drinfeld Double of the finite group $G$ defined in \cite[Section 3.2]{Dijkgraaf}. 

\end{rem}

We finish with a corollary of the results of this appendix and the ones of section \ref{subsection: G actiing on itself}. A quasi-isomorphism between  coquasi- bialgebras $(H,\psi)$ and $(H',\psi)$ is a pair $(F,\theta)$ consisting of a coalgebra isomorphism $F:H\to H'$ and a convolution invertible map $\theta\in (H\otimes H)^*$ such that
$\theta(1\otimes h)=\theta(h\otimes 1)=\varepsilon(h)$ and $$\theta(g_{1}\otimes h_1)F(g_2 h_2)=F(g_1)F(h_2) \theta(g_2\otimes h_2),$$ for all $g,h\in H$. A quasi-isomorphism is called an isomorphism of coquasi-bialgebras if additionally  $$\theta(f_1\otimes g_2)\theta(f_2g_2\otimes h_1)\psi(f_3\otimes g_3\otimes h_3)=\psi'(F(f_1)\otimes F(g_1)\otimes F(h_1))\theta(g_2\otimes h_2)\theta (f_2\otimes g_3h_3).$$

In general,  the tensor category of representations of $D^w(G)$ is not equivalent to the category of representations of any Hopf algebra. However, by the isomorphism outlined in \eqref{equation final remark}, if $G$ is finite and acts over itself by conjugation, the Grothendieck ring of $Bun_G(\mathcal K)$ for any 3-cocycle in $Z^3(\text{Tot}^*(A^{*,*}(G\rtimes G,\Tt)))$, is always equivalent to the Grothendieck ring of  $Bun_G(\mathcal K')$, where the 3-cocycle associated lives in $Z^3(\text{Tot}^*(B^{*,*}(G\rtimes G,\Tt)))$. By Theorem  \ref{appen theorem 1}, $Bun_G(\mathcal K')$ is the category of representation of a Hopf algebra, so in particular we can conclude
\begin{cor} 
The Twisted Drinfeld Double of a finite group is always  quasi-isomorphic to a Hopf algebra.
\end{cor}

\def\cprime{$'$}


\end{document}